\documentclass[10pt,letterpaper]{article}
\usepackage[margin=1in]{geometry}
\usepackage{amsmath}
\usepackage{amsthm}
\usepackage{amsfonts}
\usepackage{cleveref}
\usepackage{enumitem}
\usepackage[sort&compress,numbers]{natbib}

\usepackage{tikz}
\usetikzlibrary{calc,fadings,decorations.pathreplacing,arrows,
	datavisualization.formats.functions,shapes.geometric}

\usepackage{tikz-network}

\def\boxit#1{\vbox{\hrule\hbox{\vrule\kern6pt
			\vbox{\kern6pt#1\kern6pt}\kern6pt\vrule}\hrule}}

\def\bse{\begin{eqnarray*}}
	\def\ese{\end{eqnarray*}}
\def\be{\begin{eqnarray}}
	\def\ee{\end{eqnarray}}
\def\bq{\begin{equation}}
	\def\eq{\end{equation}}
\def\bse{\begin{eqnarray*}}
	\def\ese{\end{eqnarray*}}

\newcommand{\bbR}{\mathbb{R}}

\newcommand{\bbN}{\mathbb{N}}

\newcommand{\bbC}{\mathbb{C}}

\newcommand{\bbZ}{\mathbb{Z}}

\newcommand{\bx}{\mathbf{x}}
\newcommand{\by}{\mathbf{y}}
\newcommand{\bv}{\mathbf{v}}
\newcommand{\bz}{\mathbf{z}}

\newcommand{\bbe}{\mathbf{e}}

\newcommand{\br}{\boldsymbol{r}}

\newcommand{\mcB}{{\mathcal B}}
\newcommand{\mcC}{{\mathcal C}}
\newcommand{\mcD}{\mathcal D}
\newcommand{\mcE}{{\mathcal E}}
\newcommand{\mcF}{\mathcal{F}}
\newcommand{\mcH}{\mathcal{H}}
\newcommand{\mcI}{{\mathcal I}}

\newcommand{\mcL}{{\mathcal L}}

\newcommand{\mcP}{{\mathcal P}}

\newcommand{\mcS}{\mathcal{S}}

\newcommand{\ii}{\mathbf{i}}

\newcommand{\pp}{\mathbf{p}}

\newcommand{\mm}{\mathbf{m}}

\newcommand{\oone}{\boldsymbol{1}}

\newcommand{\Nset}{\mathbb{N}_0}

\newcommand{\Pol}{\mathbb{P}}
\newcommand{\eset}[1]{{\mathbb E} \left[ #1 \right] }

\newcommand{\dist}{\operatorname{dist}}

\newcommand{\Real}{\mathop{\text{\rm Re}}}
\newcommand{\Imag}{\mathop{\text{\rm Im}}}

\definecolor{darkgreen}{rgb}{0, 0.6, 0}
\definecolor{airforceblue}{rgb}{0.36, 0.54, 0.66}
\definecolor{applegreen}{rgb}{0.55, 0.71, 0.0}
\definecolor{asparagus}{rgb}{0.53, 0.66, 0.42}
\definecolor{cadetblue}{rgb}{0.37, 0.62, 0.63}
\definecolor{cambridgeblue}{rgb}{0.64, 0.76, 0.68}
\definecolor{olivine}{rgb}{0.6, 0.73, 0.45}
\definecolor{rufous}{rgb}{0.66, 0.11, 0.03}
\definecolor{sangria}{rgb}{0.57, 0.0, 0.04}
\definecolor{neworange}{rgb}{1, 0.64, 0}
\definecolor{flowblue}{rgb}{0.4471,    0.6235,    0.8118}
\definecolor{lightsteelblue}{RGB}{176,196,222}
\definecolor{brownblue}{RGB}{222, 202, 176}

\newtheorem{defi}{Definition}
\newtheorem{theo}{Theorem}
\newtheorem{prop}{Proposition}

\newtheorem{cor}{Corollary}
\newtheorem{prob}{Problem}

\newtheorem{asum}{Assumption}

\theoremstyle{remark}
\newtheorem{rem}{Remark}

\colorlet{fillcolor1}{lightsteelblue}
\colorlet{fillcolor2}{lightsteelblue!150}
\colorlet{fillcolor3}{lightsteelblue!200}
\colorlet{outlinecolor}{olivine!50!black!}
\colorlet{outlinecolor}{brownblue!50!black!}

\title{Uncertainty quantification and complex analyticity of the nonlinear Poisson-Boltzmann equation for the interface problem with random domains}

\author{Trevor Norton, Jie Xu, Brian Choi, Mark Kon, Julio Enrique Castrill\'on-Cand\'as}

\date{}
\crefname{prop}{proposition}{propositions}
\crefname{asum}{assumption}{assumptions}

\begin{document}
	
	\maketitle
	
\begin{abstract}
    
The nonlinear Poisson-Boltzmann equation (NPBE) is an elliptic partial differential equation used in applications such as protein interactions and biophysical chemistry (among many others). It describes the nonlinear electrostatic potential of charged bodies submerged in an ionic solution. The kinetic presence of the solvent molecules introduces randomness to the shape of a protein, and thus a more accurate model that incorporates these random perturbations of the domain is analyzed to compute the statistics of quantities of interest of the solution. When the parameterization of the random perturbations is high-dimensional, this calculation is intractable as it is subject to the curse of dimensionality. However, if the solution of the NPBE varies analytically with respect to the random parameters, the problem becomes amenable to techniques such as sparse grids and deep neural networks. In this paper, we show analyticity of the solution of the NPBE with respect to analytic perturbations of the domain by using the analytic implicit function theorem and the domain mapping method. Previous works have shown analyticity of solutions to linear elliptic equations but not for nonlinear problems. We further show how to derive \emph{a priori} bounds on the size of the region of analyticity. This method is applied to the trypsin molecule to demonstrate that the convergence rates of the quantity of interest are consistent with the analyticity result. Furthermore, the approach developed here is sufficiently  general enough to be applied to other nonlinear problems in uncertainty quantification.  

\end{abstract}



\section{Introduction}

Nonlinear elliptic partial differential equations (PDEs) are frequently used as models for applications in electrostatics. In particular, a salient problem in the field is the modeling of potential fields generated by molecules in solvents. The nonlinear Poisson-Boltzmann Equation (NPBE) serves as an accurate representation of the molecule-solvent interactions and is employed in molecular dynamics simulations and chemical applications \cite{Gray2018, Stein2019}. It has been used in the modeling of electrode-electrolyte interfaces \cite{Sundararaman2017, Sundararaman2018, Nattino2019} and in solvers such as the Adaptive Poisson–Boltzmann Solver (APBS) in determining the electrostatic potential for biomolecular processes \cite{Baker2001,jurrus2018improvements}.

The NPBE also finds applications in other various scientific disciplines.
It has been studied in many fields such as
applied mathematics \cite{Cai2013, Rubinstein1990, Li2009}, 
biophysical chemistry \cite{Ohshima2010, Edsall1958}, 
biochemistry \cite{Bergethon1998}, 
chemical physics \cite{Frenkel1946, Kirkwood1961}, 
colloids \cite{Butt2013, Israelachvili1991, Verwey1948, Evans1999, Hunter2001, Lyklema1991}, 
condensed matter physics \cite{Chaikin1995}, 
electrochemistry \cite{Schmickler2010, Bockris1970, Sparnaay1972}, 
electrolyte solutions \cite{Barthel1998, Friedman1962}, 
liquid state theory \cite{Hansen2013, Fawcett2004, March1984}, 
many-body theory \cite{Brout1963, Giuliani2003}, 
materials science \cite{Chavazaviel1999}, 
medical physics \cite{Hobbie1988}, 
molecular biology \cite{Sneppen2005}, 
physiology \cite{Bayliss1959, Hober1947}, 
physical chemistry \cite{Berry2000, Atkins1986}, 
plasma physics \cite{Morrison1967, Ichimaru1984},
polymer physics \cite{Muthukumar2011}, 
soft matter \cite{Doi2016, Dean2014, Holm2001, Poon2006}, 
solid-state physics \cite{Ashcroft1976, Kittel1996, McKelvey1966}, 
statistical mechanics \cite{Blum1992, Landau1958, McQuarrie1976}, 
surface science \cite{Israelachvili1991, Butt2013}, 
thermodynamics \cite{Glasstone1947, Lewis1961}, 
among others.

The NPBE can be written as 

\begin{equation}\label{npbe}
\begin{aligned}
	-\nabla \cdot (\epsilon(\bx) \nabla u(\bx)) + \overline{\kappa}^2(\bx) \sinh(u(\bx)) &= f(\bx),  &\text{for }\bx &\in \mcD, \\
	u(\bx) &= g(\bx), &\text{for } \bx &\in \partial \mcD,
\end{aligned}
\end{equation}
where \(\mcD\subset \bbR^3\) is the domain, \(\epsilon(\bx)>0\) is a dimensionless dielectric function, \(\overline \kappa(\bx)\geq 0\) is the modified Debye-H\"uckel parameter, and \(f(\bx)\) gives the charge of the particles in the region. The desired solution \(u\) represents the (dimensionless) potential function. One typically assumes the domain is separated into three parts: the solvent, the molecular region, and the ion-exclusion layer. From this assumption, \(\epsilon(\bx)\) becomes discontinuous at the interfaces between these regions, and so this PDE is sometimes referred to as an \emph{elliptic interface problem}. Variational methods can be used to show that a unique weak solution (i.e.\ a function \(u\in H^1\)) exists under certain conditions \cite{holst1994poisson}.

When computing molecular dynamics (MD), the presence of  thermal fluctuations and solvent interactions (among other factors) can lead to random conformations of the molecules, and more accurate models incorporate this stochasticity. For instance, stochastic initial velocities are used in \cite{Allen2004,Neumaier1997} when computing MD. Other approaches to MD which factor in stochasticity include Langevin dynamics \cite{Ceriotti2009, Hanggi1995, Jung1987} and Markov random models \cite{Xia2013}. In this paper we assume that the random domain conformations are represented using a finite dimensional model with $N$ random variables such as Karhunen-Lo\`{e}ve expansions or similar random field stochastic representations.

Quadrature methods can compute stochastic measures for the Quantity of Interest (QoI) of the potential field under random configurations of the molecules. However, for each quadrature point we must compute the solution of the NPBE. As the number of dimensions, $N$, increases the calculation quickly becomes intractable.

One strategy to reducing the cost of the calculations is to show that the QoI varies analytically with respect to the stochastic parameters. In this case, one can compute the \(N\)-dimensional quadrature using a sparse grid \cite{nobile2008a}, which gives sub-exponential or algebraic decay in the error as a function of the number of interpolation points. Thus the ``curse of dimensionality'' from the \(N\) can be ameliorated, and the problem becomes tractable. If the QoI depends analytically on the solution \(u\), then it is sufficient to prove that the solution varies analytically with respect to the stochastic parameters.

Previous studies in uncertainty quantification (UQ) have explored the analyticity of solutions to linear partial differential equations with random domains \cite{castrillon2016,castrillon2021stochastic}. The authors in \cite{Heitzinger2018} explore the utilization of stochastic collocation and Galerkin methods for the NPBE. The NPBE is treated as a semi-linear stochastic boundary valued problem and the existence of a unique solution is proved. This approach extends the existence and uniqueness result found in the deterministic case \cite{holst1994poisson} and follows a similar proof strategy. However, due to the absence of analytic regularity results, convergence rates for implementing the stochastic collocation method are not derived. For the case of elliptic interface problems, the regularity of point evaluations of solutions and how to approximate them with Deep Neural Networks (DNNs) have been studied \cite{scarabosio2022deep}. Furthermore, in \cite{Opschoor2021} the authors show  that given a holomorphic map, such as an analytic extension of the solution of a PDE, there exists a DNN with exponential accuracy with respect to the dimensionality of the DNN.

In the case of the NPBE, our objective is to demonstrate that analytical deformations of the domain result in the analytic variation of the solution \(u\). This particular investigation introduces two challenges that were not addressed in previous research:

\begin{enumerate}[label=(\arabic*)]
\item \emph{Nonlinearity}. Previous results have been proved by showing that the solutions satisfy the Cauchy equations and are thus analytic. This becomes difficult to do with the nonlinearity introduced by \(\sinh(u)\). A more subtle complication due to the nonlinearity is that outputs of the function might not end up in the desired space. That is, for potential weak solutions \(u \in H^1\), it is not guaranteed that \(\sinh(u)\in L^2\).
\item \emph{Interfaces.}  In \cite{choi2021existence} the analyticity properties of the NPBE are studied where $\epsilon$ exhibits some degree of regularity. However in our case, the assumption that $\epsilon$ is Lipschitz continuous on the entire domain is relaxed to account for the interface problem.  
\end{enumerate}

The main strategy of this paper is to use the implicit function theorem and the domain mapping method (introduced in \cite{castrillon2016analytic}) to show that \(u\) is analytic with respect to the stochastic parameters. This avoids trying to show the Cauchy equations hold, and it is a general strategy that can be applied to other nonlinear PDEs. In order to apply the implicit function theorem, we have to specify a function domain for our solution. As noted above, we cannot take \(u\in H^1\) since this does not imply \(\sinh(u)\) is in \(L^2\). If \(u\) were in \(H^2\) then the fact that the Sobolev space is a Banach algebra (see \cite[Thm.~4.39]{adams2003sobolev} for a proof) would let us conclude the nonlinearity is in \(L^2\), but the discontinuity of \(\epsilon\) in the problem means that \(u\) is not (weakly) differentiable across the interfaces. However, we can instead define a ``piecewise \(H^2\)'' space, in which \(u\) naturally lies. Thus we can then apply the implicit function theorem to get analyticity. From there, estimates of the rate of convergence of the sparse grid method can be obtained by getting \emph{a priori} bounds on the region of analyticity of the solution. The results here are also notable in that they can easily be generalized to other UQ problems that come from nonlinear PDEs with interfaces.

The paper will be structured as follows. In \cref{linear-pde-section}, we introduce the problem of a linear elliptic PDE with interfaces. We introduce a suitable Banach space in which a strong solution naturally exists: a ``piecewise \(H^2\)'' space denote by \(\mcH(U)\). It is then shown that linear problem has unique strong solutions that induces an isomorphism between \(\mcH(U)\) and the Banach space of the forcing functions. In \cref{npbe-reference-domain-section}, the NPBE is reformulated onto a reference domain with solutions in \(\mcH(U)\). From there, the implicit function theorem is used to get an analytic mapping from the parameter space to the solutions of the NPBE. Furthermore, we give details on how to get \emph{a priori} bounds on the size of the region of analyticity after applying the implicit function theorem. \Cref{sparsegrids} gives an overview of applying sparse grids to the efficient computation of integrals of analytic functions. Finally, numerical experiments are performed in \cref{numerics} to demonstrate the convergence results.

\section{The Linear Elliptic PDE with Interfaces}\label{linear-pde-section}

\subsection{Definitions and notations}

We first consider a linear elliptic PDE with possibly discontinuous coefficients at interfaces. For our problem, the domain is split up into three subdomains. The coefficients of the PDE are sufficiently regular on each subdomain, and the subdomains are nested within each other. The boundary between the subdomains form the interfaces in our problem. It is straightforward to generalize this to an arbitrary number of nested subdomains.

\begin{defi}
We say that a connected, bounded open set \(U\subset \bbR^3\) is \textbf{properly decomposed} into \(l\) subdomains \(U_1,U_2,\ldots, U_l\) if the following holds: There is a sequence of compactly embedded subsets \[U^{(l-1)} \subset\subset U^{(l-2)} \subset\subset \cdots \subset \subset U^{(1)} \subset\subset U \] where 
\begin{align*}
	U_l &= U \setminus \overline{U^{(1)}} \\
	U_{l-1} &= U^{(1)} \setminus \overline{U^{(2)}} \\
	&\hspace{0.5em}\vdots \\
	U_2 &= U^{(l-2)} \setminus \overline{U^{(l-1)}} \\
	U_1 &= U^{(l-1)}.
\end{align*}
We define the interfaces \(I_1, \ldots, I_{l-1}\) to be \(I_i = \partial U_i \cap \partial U_{i+1}\) for \(i = 1, \ldots, l-1\).
\end{defi}

\begin{figure}[ht]
\centering
\begin{tikzpicture}[scale=0.85]
	\draw[fill=fillcolor1, draw=outlinecolor] plot [smooth cycle] coordinates {(1,2) (2,6) (1.8, 8) (5,10) (9,9) (10,6) (8,3) (6,1) (3,1)};
	\draw node at (2,2.2) {\LARGE$U_3$};
	\draw[fill=fillcolor2, draw=outlinecolor] plot [smooth cycle] coordinates {(3,3) (2.74,4.5) (3.25, 6.7) (4, 8.5) (5.4, 9) (7, 8.9) (7.5, 8) (6.5, 5) (7,3) (4,2)};
	\draw node at (4,3) {\LARGE$U_2$};
	\draw[fill=fillcolor3,draw=outlinecolor] plot [smooth cycle] coordinates {(4,4) (4,6.5) (5.5, 7.2) (6, 7) (5.7,5.5) (5,4)};
	\draw node at (4.75,5.5) {\LARGE$U_1$};
	\draw node at (1.4, 8.75) {\Large \color{outlinecolor}$I_3 = \partial U$};
	\draw node at (3, 7.75) {\Large \color{outlinecolor}$I_2$};
	\draw node at (4.2, 7.2) {\Large \color{outlinecolor}$I_1$};
\end{tikzpicture}
\caption{Here a set \(U\) is properly decomposed into the three subdomains \(U_1\), \(U_2\), and \(U_3\). The interface \(I_1\) is where the boundaries of \(U_1\) and \(U_2\) meet, and the interface \(I_2\) is where the boundaries of \(U_2\) and \(U_3\) meet. The boundary \(\partial U\) is also referred to as \(I_3\).} 
\end{figure}
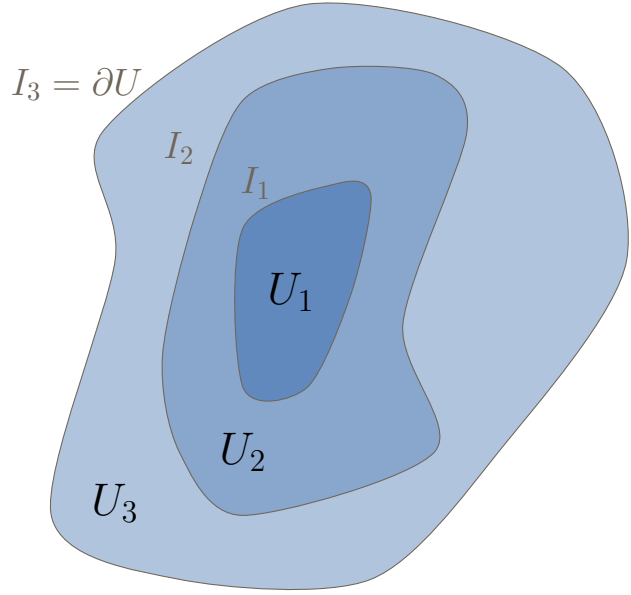

\begin{rem}
The results for this section could also be generalized to the case where subdomains are no longer strictly nested within each other. However, we will use the above definition since it is sufficient for our application and it keeps the notation simple.
\end{rem}

For convenience, we shall refer to \(\partial U\) by \(I_l\). We now assume that \(U\subset \bbR^3\) is properly decomposed into \(l\) subdomains \(U_1,\ldots, U_l\) where the interfaces and the boundary of \(U\) are all of class \(C^{1,1}\). We choose our domain to be in \(\bbR^3\) for our application; other dimensions are possible, but the choices of Sobolev spaces will be affected. Denote by \(\nu_k\) the outward facing normal for the surface \(I_k\). On each of these surfaces we can define trace operators. For \(1\leq k \leq l-1\), there are two trace operators depending on if we take the domain to be \(H^1(U_k)\) or \(H^1(U_{k+1})\). Let 
\begin{equation*}
	\gamma^+_k : H^1(U_k) \to H^{1/2}(I_k)
\end{equation*}
be the trace operator from the domain \(U_k\) to its outer boundary (for \(1\leq k \leq l\)), and similarly let
\begin{equation*}
	\gamma^-_k : H^1(U_{k+1}) \to H^{1/2}(I_k)
\end{equation*}
be the trace operator from the domain \(U_{k+1}\) to its inner boundary (for \(1\leq k\leq l-1\)).

Define a second-order elliptic operator \(\mcP_k\) on each \(U_k\) by
\begin{equation}\label{second-order-operator-k}
	\mcP_k u := - \sum_{i=1}^3 \sum_{j=1}^3 \partial_i(a_{ij} \partial_j u) + c u
\end{equation}
where \(a_{ij} \in C^{0,1}(\overline{U_k})\) for each \(i,j=1,2,3\) and \(k=1,2,\ldots, l\) and \(c \in L^\infty(U)\). We assume that \(a_{ij} = a_{ji}\) for all \(i\) and \(j\). We further assume that each \(\mcP_k\) satisfies a uniform ellipticity condition on \(U_k\), i.e., there exists a constant \(\theta > 0\) such that 
\begin{equation*}
	\sum_{i,j=1}^3 a_{ij}(x) \xi_i\xi_j \geq \theta |\xi|^2
\end{equation*}
for a.e.\ \(x\in U_k\) and all \(\xi \in \bbR^3\). We can choose \(\theta\) independently of \(k\). The operator \(\mcP_k\) is naturally associated with the bilinear map \(\Phi_k\) given by
\begin{equation*}
	\Phi_k(u,v) = \int_{U_k} \left( \sum_{i=1}^3 \sum_{j=1}^3 a_{ij} \partial_j u \partial_i v + c u v \right)\, \mathrm d x
\end{equation*}

The operators \(\mcP_k\) define co-normal derivatives on the interfaces and boundary. We define \(\mcB^{\pm}_k\) by 
\begin{equation*}
	\mcB_k^\pm u := \sum_{i=1}^3 (\nu_k)_i \gamma_k^{\pm} \left( \sum_{j=1}^3 a_{ij} \partial_j u \right),
\end{equation*}
where \((\nu_k)_i\) denotes the \(i\)th component of the normal vector \(\nu_k\). This gives maps \(\mcB_k^+ : H^2(U_k)\to H^{1/2}(I_k)\) and \(\mcB_k^- :H^2(U_{k+1}) \to H^{1/2}(I_k)\). By specifying the value \(f\in H^{-1}(U_k)\) of \(\mcP_k u\), the conormal derivative can be extended to \(H^1(U_k)\) functions. If the choice of \(f\in H^{-1}(U_k)\) is clear, then we will simply say the distribution \(\mcB^\pm_k u \in H^{-1/2}(I_k)\) is the conormal derivative of \(u\).

These definitions allow us to use the following Green's identity for \(2\leq k \leq l\):
\begin{equation}\label{greens-1}
	\Phi_k(u,v) = (\mcP_k u,v)_{U_k} + (\mcB^+_k u, \gamma_k^+ v)_{I_k} - (\mcB^-_{k-1}u, \gamma_{k-1}^- v)_{I_{k-1}}, \quad \text{for all } u\in H^2(U_k), v\in H^1(U_k).
\end{equation}
And in the case of \(k=1\) we have 
\begin{equation}\label{greens-2}
	\Phi_1(u,v) = (\mcP_1 u,v)_{U_1} + (\mcB^+_1 u, \gamma_1^+ v)_{I_1}, \quad \text{for all } u\in H^2(U_1), v\in H^1(U_1).
\end{equation}

\subsection{Weak and strong forms of elliptic problem with interfaces}\label{subsection-linear-elliptic-problems}

Similar to standard elliptic PDE theory, the existence and uniqueness of the weak solution for the discontinuous interface problem 
will first be established. Consequently, this result will be used to show existence and uniqueness of the strong solution. However, to motivate the weak formulation, we first start off by defining the strong form of the problem.

Normally, the strong solution of an elliptic PDE would lie in the space \(H^2(U)\), but this cannot be the case for our problem since we can lose regularity at the interfaces. The next best option is to require a ``piecewise \(H^2\) regularity'' for the strong solution, where the function is \(H^2\) when restricting to the subdomains.

\begin{defi}
	Let 
	\begin{equation*}
		\mcH(U;U_1,U_2,\ldots, U_l) := \{u \in H^1(U) \mid u|_{U_k} \in H^2(U_k) \text{ for } k=1,2,\ldots, l\}.
	\end{equation*}
	This is a Banach space with a norm given by
	\begin{equation*}
		\| u \|_{\mcH} = \|u\|_{H^1(U)} + \sum_{k=1}^l \left\| u|_{U_k} \right\|_{H^2(U_k)}.
	\end{equation*}
\end{defi}
This Banach space depends on our decomposition of \(U\), but when this decomposition is clear we will simply write \(\mcH(U)\) instead of \(\mcH(U;U_1,U_2,\ldots, U_l)\). Throughout the paper, we will denote \(u|_{U_k}\) or \(f|_{U_k}\) by \(u_k\) or \(f_k\), respectively, to cut down on notation.

Requiring the strong solution \(u\) to lie in \(\mcH(U)\) 
is insufficient to define a unique strong solution for the elliptic problem. If the strong solution were only required to satisfy \(\mcP_k u_k = f_k\) on each \(U_k\) along with a Dirichlet boundary condition, then infinitely many solutions would be possible; for instance, in the case where \(c \equiv 0\) adding a constant value to \(u_1\) on \(U_1\) would give another solution to the problem. Unique solutions exist if certain jump conditions are satisfied at the interfaces.

For \(u\in \mcH(U;U_1,U_2,\ldots,U_l)\), define 
\begin{equation*}
	[\mcB_k u]_{I_k} = \mcB_k^+u_k - \mcB^-_k u_{k+1} \in H^{1/2}(I_k).
\end{equation*}
Then the strong form of the elliptic PDE with interfaces is stated as follows:

\begin{prob}[Strong form of elliptic PDE with interfaces]\label{strong-linear-interface-problem}
	Suppose \(U\) is properly decomposed into subdomains \(U_1,U_2,\ldots,U_l\), where the interfaces and boundary are of class \(C^{1,1}\). Let \(\mcP_k\) be defined as in \cref{second-order-operator-k}. Fix \(f \in L^2(U)\), \(g_k\in H^{1/2}(I_k)\) for \(k=1,2,\ldots, l-1\), and \(g_l \in H^{3/2}(\partial U)\). A function \(u \in \mcH(U;U_1,U_2,\ldots, U_l)\) is a strong solution of the elliptic PDE with interfaces if 
	\begin{align}
		\mcP_k u_k &= f_k & &\text{for } k=1,2,\ldots, l, \\
		\left[\mcB_k u\right]_{I_k} &= g_k & &\text{for } k = 1,2,\ldots, l-1, \\
		\gamma^+_l u_l &= g_l. 
	\end{align}
\end{prob}

The boundary condition can be set to zero by setting \(w\in H^2(U)\) to be such that \(\gamma_l^+ w_l = g_l\). Then we can break up \(u\) into \(u = \tilde u + w\) where \(\tilde u \in H^1_0(U)\cap \mcH(U)\). The weak formulation of the problem is derived by taking
\begin{equation*}
	\mcP \tilde u = f - \mcP w
\end{equation*}
(where \(\mcP\) is the differential operator that is locally \(\mcP_k\) on each \(U_k\)), multiplying each side of the equation by \(v\in H^1_0(U)\), and integrating over \(U\). Applying \cref{greens-1,greens-2} and summing up the terms gives us 
\begin{equation}\label{formal-integration}
	\sum_{k=1}^l \Phi_k(\tilde u_k , v_k) = \left(\sum_{k=1}^l  ( f_k, v_k)_{U_k}  - (\mcP_k w_k, v_k)_{U_k} \right)+  \left(\sum_{k=1}^{l-1} (g_k, \gamma_k^+ v_k)_{I_k} - ([\mcB_k w]_{I_k} , \gamma_k^+ v_k)_{I_k}\right).
\end{equation}
Note that since \(v\in H^1_0(U)\), the functions \(\gamma^+_k v_k\) and \(\gamma^-_{k+1} v_{k+1}\) will be equal, which allowed us to combine terms in the equation above. \Cref{formal-integration} makes sense even in the case where \(\tilde u\) is not in \(\mcH(U)\), so we use this equation to define a weak solution in \(H^1_0(U)\).

\begin{prob}[Weak form of elliptic PDE with interfaces]\label{weak-linear-interface-problem}
	Suppose \(U\) is properly decomposed into subdomains \(U_1,U_2,\ldots,U_l\), where all interfaces and the boundary are of class \(C^{1,1}\). Let \(\mcP_k\) be defined as in \cref{second-order-operator-k}, and let \(\Phi_k\) be the bilinear maps associated with \(\mcP_k\). Fix \(f \in L^2(U)\), \(g_k\in H^{1/2}(I_k)\) for \(k=1,2,\ldots, l-1\), and \(g_l \in H^{3/2}(\partial U)\). Take \(w \in H^2(U)\) so that \(\gamma_l^+w = g_l\). A function \(u = \tilde u + w\) with \(\tilde u \in H^1_0(U)\) is a weak solution to the elliptic PDE with interfaces if
	\begin{equation}\label{weak-eqn}
		\sum_{k=1}^l \Phi_k(\tilde u_k , v_k) = \left(\sum_{k=1}^l  ( f_k, v_k)_{U_k}  - (\mcP_k w_k, v_k)_{U_k} \right)+  \left(\sum_{k=1}^{l-1} (g_k, \gamma_k^+ v_k)_{I_k} - ([\mcB_k w]_{I_k} , \gamma_k^+ v_k)_{I_k}\right), \quad \forall v \in H^1_0(U).
	\end{equation}
\end{prob}

\begin{rem}
	The formulation of \cref{weak-linear-interface-problem} agrees with the weak form of the linear Poisson Bolztmann equation in the case where \(g_1,g_2,\ldots, g_{l-1}\) are set to zero (c.f.\ \cite{holst1994poisson}), which suggests that this is the appropriate formulation of the weak problem for our application. Although in practice there will be no forcing terms on the interfaces, allowing the possibility of non-zero \(g_k\)'s is of theoretical importance when we later apply the implicit function theorem.
\end{rem}

Showing that \cref{weak-linear-interface-problem} has unique solutions follows from applying the Lax-Milgram theorem in a similar way to how it is applied in standard linear elliptic theory.  

\begin{prop}\label{existence-uniqueness-linear-weak-solns}
	Suppose \(c\in L^\infty(U)\) is non-negative. Then we have a unique solution to \cref{weak-linear-interface-problem}.
\end{prop}

\begin{rem}
	The regularity of the data can be loosened in \cref{existence-uniqueness-linear-weak-solns}; for instance, we can let \(f\in H^{-1}(U)\) and still have unique weak solutions. However, for our purposes we do not need these low regularity results. 
\end{rem}

The uniqueness of weak solutions can now be used to show the existence and uniqueness of strong solutions.

\begin{theo}\label{linear-regularity-weak-soln}
	Suppose that \(c\in L^\infty(U)\) is non-negative. Then there exists a unique solution \(u\in \mcH(U)\) to \cref{strong-linear-interface-problem}. Moreover, \cref{strong-linear-interface-problem} defines an isomorphism between \(\mcH(U)\) and \(L^2(U) \times \prod_{k=1}^{l-1} H^{1/2} (I_k) \times H^{3/2}(I_l)\) by associating solutions \(u\) with data \((f,g_1,g_2,\ldots, g_{l-1}, g_l)\).
\end{theo}

\begin{proof}
	From \cref{existence-uniqueness-linear-weak-solns}, it follows that the weak solution \(u \in H^1(U)\). To show that \(u \in \mcH(U)\), we can appeal to \cite[Thm.~4.20]{mclean2000strongly}. Namely we have that 
	\begin{equation*}
		\begin{aligned}
			u_k &\in H^1(U_k), &  &\text{for } k=1,2,\ldots, l, \\
			[\mcB_k u]_{I_k} &\in H^{1/2}(I_k), &  &\text{for } k=1,2,\ldots,l-1, \text{ and } \\
			\gamma_l^+ u_l &\in H^{3/2}(I_l),
		\end{aligned}
	\end{equation*}
	which implies that \(u_k \in H^2(U_k)\) for each \(k=1,2, \ldots, l\) and thus \(u \in \mcH(U)\). Applying the Green's identities in  \cref{greens-1,greens-2} demonstrates that \(u\) is a solution of \cref{strong-linear-interface-problem}. Since solutions of \cref{strong-linear-interface-problem} are also solutions of \cref{weak-linear-interface-problem}, the strong solution \(u\) is unique. The isomorphism result follows immediately from existence and uniqueness of the strong solution and the continuity of the solution map with respect to the boundary and forcing data.
\end{proof}

\section{The Nonlinear Poisson Bolztmann Equation on a Reference Domain}\label{npbe-reference-domain-section}

\subsection{Existence of Region of Analyticity}

We now return to our main focus of the paper: solutions for the NPBE. The NPBE given in \cref{npbe} is a nonlinear elliptic PDE. For our applications, we will have the main domain properly decomposed into three subdomains. We also allow for the possibility of random perturbations of the boundary and interfaces. 

Let \(\Omega\) be the sample space. Each outcome \(\omega \in \Omega\) designates a random domain \(\mcD(\omega)\) on which the NPBE will evaluated. The domain \(\mcD(\omega)\) is properly decomposed into three subdomains \(\mcD_1(\omega)\), \(\mcD_2(\omega)\), and \(\mcD_3(\omega)\) with interfaces \(\mcI_1(\omega)\) and \(\mcI_2(\omega)\). The parameters \(\epsilon\), \(\overline{\kappa}^2\), \(g\), and \(f\) will also depend on \(\omega\). From here one can define strong and weak solutions of the NPBE on the stochastic domain in a similar way to that in \cite{castrillon2016analytic}. In practice, one usually assumes that the value of each parameter is given as a function of the random vector \(\mathbf Y(\omega) = (Y_1(\omega), Y_2(\omega),\ldots, Y_N(\omega))\) taking values on the compact set \(\Gamma \subset \bbR^N\) and with known density  \(\rho :\Gamma \to \bbR_{\geq0}\). Typically \(\Gamma = [-1,1]^N\) with \(\rho\) a truncated normal distribution, although the distribution can be more general. Often the parameters will vary analytically with respect to the value of \(\mathbf Y\) and are usually polynomials in \(\mathbf Y\). Thus the NPBE can be stated as a problem with parameters \(\by \in \Gamma \subset \bbR^N\).

To parameterize the random domain, we assume that the random domain has a pullback onto some fixed open set for each \(\omega\). In particular, take \(U\subset \bbR^3\) to be a bounded, open set that is properly decomposed into three subdomains \(U_1\), \(U_2\), and \(U_3\) with interfaces \(I_1 = \partial U_1 \cap \partial U_2\) and \(I_2 = \partial U_2 \cap \partial U_3\). The interfaces and boundary are taken to have \(C^{1,1}\) regularity. We assume for each \(\by \in \Gamma\) that there is \(F(\cdot ; \by ) \in C^2_{\mathrm{diff}}(\bbR^3,  \bbR^3)\) such that 
\begin{equation*}
	F(U_k;\by) = \mcD_k(\by) \quad \text{for } k = 1,2,3
\end{equation*}
and 
\begin{equation*}
	F(I_k;\by) = \mcI_k(\by) \quad \text{for } k = 1,2.
\end{equation*}
The Jacobian matrix of \(F(\cdot; \by)\) will be denoted by \(J(\cdot;\by)\). Note that since \(F\) is a \(C^2\) diffeomorphism, the regularity of the interfaces and the boundary are preserved under the mapping. To distinguish between the coordinates in each domain, we will denote by \(\br\) elements of \(U\) and by \(\bx\) elements of \(\mcD(\by)\). Similarly \(\nabla_{\br}\) and \(\nabla_{\bx}\) will be used to distinguish between the derivatives in \(U\) and \(\mcD(\by)\), respectively. 	

Hence, we can define the strong form of the NPBE on the random domain. Again, we will use subscripts to denote restrictions to the subdomains (e.g.\ \(u_k = u |_{\mcD_k}\)). The trace operators \(\gamma_k^{\pm}\) are defined similar to those in \cref{subsection-linear-elliptic-problems}. Also, we have the conormal derivative \(\mcB_k\) using the elliptic operator defined by \(a_{ij} (\bx;\by) = \epsilon(\bx;\by) \delta_{ij}\), where \(\delta_{ij}\) is the Kronecker delta.

\begin{prob}\label{strong-npbe-random-domain}
	The function \(\Gamma \ni \by \mapsto u(\cdot; \by) \in \mcH(\mcD(\by);\mcD_1(\by), \mcD_2(\by), \mcD_3(\by)) \) is a strong solution for the NPBE on the random domain if for each \(\by \in \Gamma\) we have \(u(\cdot;\by)\) satisfies
	\begin{align}
		-\nabla_\bx \cdot (\epsilon_k(\bx; \by) \nabla_\bx u_k(\bx; \by)) + \overline{\kappa}^2_k(\bx;\by) \sinh(u_k(\bx; \by)) &= f_k(\bx;\by), &  \text{for } k &= 1,2,3 \\
		[\mcB_k u(\cdot;\by)]_{\mcI_k(\by)} &= 0 & \text{for } k&=1,2 \\
		\gamma_3^+(u_3(\cdot;\by)) &= g(\cdot;\by).
	\end{align}
\end{prob}

From the strong form, we can derive a weak version of the NPBE by integrating against a test function and using integration by parts. Here we define \(w(\cdot;\by) \in H^2(\mcD(\by))\) to be the inverse trace of \(g(\cdot;\by)\) so that \(\gamma_3^+(w) = g\).

\begin{prob}\label{weak-npbe-rand-domain}
	The function \(\Gamma \ni \by \mapsto u(\cdot ;\by) = \tilde u(\cdot;\by) + w(\cdot;\by)\) is a weak solution for the NPBE on the random domain if for each \(\by\in \Gamma\) we have \(\tilde u (\cdot;\by) \in H^1_0(\mcD(\by))\) and 
	\begin{multline}
		\int_{\mcD(\by)} \epsilon(\bx;\by) \nabla_\bx \tilde u(\bx;\by) \cdot \nabla_\bx v(\bx) + \overline{\kappa}^2(\bx;\by) \sinh(\tilde u(\bx;\by) + w(\bx;\by)) \, d\bx \\
		= \int_U f(\bx;\by) v(\bx) \, d\bx - \int_{\mcD(\by)} \epsilon(\bx;\by) \nabla_\bx w(\bx;\by) \cdot \nabla_\bx v(\bx)\, d\bx , \quad \forall v \in H^1_0(\mcD(\by)).
	\end{multline}
\end{prob}

The weak form of the NPBE given above agrees with the standard definition of the weak form for this equation (c.f.\ \cite[\S 2.1.5]{holst1994poisson}). To pull back onto the reference domain, we construct an equivalent weak form of the NPBE for the pullback of \(u\) onto \(U\) given by \(u^*(\cdot;\by) = u(F(\cdot;\by);\by)\). Following results given in \cite{castrillon2016analytic,castrillon2021hybrid,castrillon2021stochastic}, we can write the weak form of the pullback onto the reference domain. 

\begin{prob}\label{weak-npbe-ref-domain}
	The function \(\Gamma \ni \by \mapsto u^*(\cdot ;\by) = \tilde u^*(\cdot;\by) + w^*(\cdot;\by)\) is a weak solution for the NPBE on the reference domain if for each \(\by\in \Gamma\) we have \(\tilde u^* (\cdot;\by) \in H^1_0(U)\) and 
	\begin{multline}\label{weak-eqn-ref-dom}
		\int_{U} \epsilon^*(\br;\by)\Big(J^{-1}(\br;\by) J^{-\mathrm{T}}(\br;\by) \det J(\br;\by) \nabla_{\br}  \tilde u^*(\br;\by) \Big)\cdot \nabla_{\br} v(\br) \\
		+ (\overline{\kappa}^2)^*(\br;\by) \sinh(\tilde u^*(\br;\by) +  w^*(\br;\by)) v(\br) \det J(\br;\by)\, d\br \\
		= \int_{U} f^* (\br;\by) v(\br)\det J(\br;\by)\, d\br \\
		- \int_U  \epsilon^*(\br;\by) \Big(J^{-1}(\br;\by) J^{-\mathrm{T}}(\br;\by) \det J(\br;\by) \nabla_{\br}  w^*(\br;\by) \Big)\cdot\nabla_{\br} v(\br)\, d\br, \quad \forall v\in H^1_0(U).
	\end{multline}
\end{prob}

The strong form of the NPBE on the reference domain is defined in an analogous way to \cref{strong-linear-interface-problem}. This also corresponds to the weak problem given in \cref{weak-npbe-ref-domain} in that assuming sufficient regularity of the weak solution and integrating by parts gives the strong formulation.
\begin{prob}\label{strong-npbe-ref-domain}
	The function \(\Gamma \ni \by \mapsto u^*(\cdot; \by) \in \mcH(U; U_1, U_2, U_3) \) is a strong solution for the NPBE on the random domain if for each \(\by \in \Gamma\) we have \(u^*(\cdot;\by)\) satisfies
	\begin{equation*}
		\begin{aligned}[]
			[\mcL_k(\by))] \big(u^*_k(\cdot;\by)\big) + (\overline \kappa^2)^*_k(\cdot;\by) \sinh(u^*_k(\cdot;\by)) \det J_k(\cdot;\by) &= f_k^*(\cdot;\by), & &\text{for } k=1,2,3, \\
			[\mcB_k u^* ]_{I_k} &= 0, & &\text{for } k =1,2, \\
			\gamma_3^+ u_3^* &= g^*,
		\end{aligned}
	\end{equation*}
	where 
	\begin{equation}\label{linear-operator}
		[\mcL_k(\by)] \big(v\big)  := -\nabla_{\br} \cdot \left( \epsilon_k^*(\cdot;\by) J_k^{-1}(\cdot;\by)J_k^{-\mathrm T}(\cdot;\by) \det J_k(\cdot;\by) \nabla_{\br} v \right), \quad \text{for } v \in H^2(U_k).
	\end{equation}
\end{prob}

To summarize, we have four problems that we can consider for the NPBE depending on whether we want the weak or strong solution and whether the domain is random or fixed. We can move from the the random domain problems to the reference domain problems by using the pullback \(F^*\). Transitioning between weak and strong forms is done by integrating-by-parts or having increased regularity of the solution. \Cref{diagram-of-problems} illustrates the relationship between these problems. We will be working with the reference domain moving forward: first showing that a weak solution exists, and then applying that result for the strong solution. We can recapture results on the random domain by composing solutions on the reference domain with the diffeomorphism \(F\).

\begin{figure}[h]
	\centering

    \begin{tikzpicture}[every label/.append style={text width=1.5cm,align=center},multilayer=3d]

        \SetLayerDistance{-2.75}          
        \begin{Layer}[layer=1]

            \Plane[x=-.5,y=-0.75,width=2.75,height=3.75]
            \node at (-.5,-.5)[below right]{$\mcD(\by)$};
            \node at (1.1,0.05)[align=center]{\small Int-by\\ \small  parts};
            \node at (1,2)[align=center]{\small Reg. results};

        \end{Layer}
        \begin{Layer}[layer=2]
            \Plane[x=-.5,y=-0.75,width=2.75,height=3.75,color=asparagus]
            \node at (-.5,-.5)[below right]{$U$};
            \node at (1,0.15)[align=center]{\small Int-by\\ \small  parts};
            \node at (1,2.05)[align=center]{\small Reg. results};
        \end{Layer}
        
        \Vertices{Network/vertices.csv}
        \Edges[Direct]{Network/edges.csv}
	\end{tikzpicture}
	\caption{\label{diagram-of-problems} Moving from the random domain to the reference domain is done by using the pullback \(F^*(\by)\). Using integration-by-parts with tests functions allows us to get the weak formulation of the NPBE from the strong formulation. If the solution of the weak problem is sufficiently regular, then it will also be a strong solution.}
\end{figure}
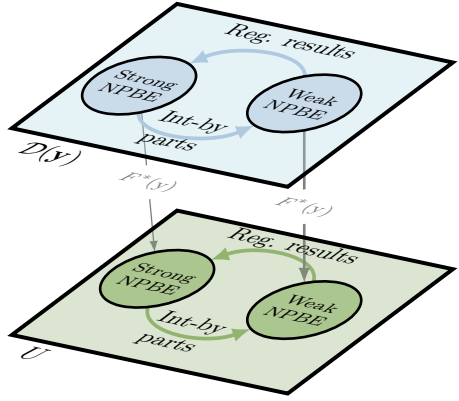

To get solutions for the NPBE, we must make some assumptions on the parameters. The assumptions that \(\epsilon^*\) and \((\overline{\kappa}^2)^*\) are positive and non-negative, respectively, come from the physics of the simulation and are generally satisfied. The function \(f^*\) is used to model the point charges, and ideally would be made to be a sum of Dirac deltas. However, there are few results for the nonlinear version of the Poisson-Boltzmann equation with forcing functions in \(H^{-2}(U)\). Thus in practice, one approximates the point charges with \(L^2\) functions (e.g.\ Gaussian functions centered at the location of the charge), and so we take \(f^*\in L^2\). The Dirichlet boundary condition is typically taken to be the long-distance approximation of the potential from the point charges and is smooth on \(\partial U=I_3\), but simply requiring \(g^*\in H^{3/2}(I_3)\) gives sufficient regularity. Finally, we assume all the parameters vary analytically with respect to \(\by\in \Gamma\), which is reasonable to assume when computing numerical solutions.

Thus we make the following assumptions:

\begin{asum}\label{asum1}
	For \(\by \in \Gamma\), we have that 
	\begin{itemize}
		\item \(\epsilon^*_k(\cdot;\by) \in C^{0,1}(\overline{U_k})\) for \(k=1,2,3\)
		\item \((\overline\kappa^2)^*(\cdot;\by)\in L^\infty(U)\)
		\item \(f^*(\cdot;\by) \in L^2(U)\) 
		\item \(g^*(\cdot;\by)\in H^{3/2}(I_3)\)
		\item \(F(\cdot;\by) \in C^2_{\mathrm{diff}}(\bbR^3,\bbR^3) \subset C^2(\bbR^3,\bbR^3)\) 
	\end{itemize}
	and the maps from \(\Gamma\) into the respective Banach spaces are analytic. Since the inverse trace operator is linear, we also have \(\by \mapsto w^*(\cdot;\by) \in H^2(U)\) is analytic.
\end{asum}
\begin{asum}\label{asum2}
	There exists \(c_1>0\) such that for any \(\by \in \Gamma\) 
	\begin{equation*}
		\epsilon^*(\br;\by)\geq c_1, \quad \forall \br \in U.
	\end{equation*}
\end{asum}
\begin{asum}\label{asum3}
	For any \(\by \in \Gamma\), we have that 
	\begin{equation*}
		(\overline\kappa ^2)^*(\br ;\by)\geq 0, \quad \text{for a.e.\ } \br \in U.
	\end{equation*}
\end{asum}

We will also need to assume that \(\det J(\br;\by)\) is bounded away from \(0\) for \(\by \in \Gamma\) and \(\br \in U\). This is essentially assuming that the map \(F\) is non-singular and preserves the orientation of the domain, both of which are reasonable when considering small perturbations of the interface.
\begin{asum}\label{asum4}
	There exists \(c_2>0\) such that for any \(\by \in \Gamma\) we have 
	\begin{equation*}
		\det J(\br;\by) \geq c_2, \quad \forall \br \in U.
	\end{equation*}
\end{asum}

Hence, we can prove the existence and uniqueness of weak solutions.
\begin{prop}\label{unique-weak-solns}
	If \cref{asum1,asum2,asum3,asum4} hold, then  \cref{weak-npbe-ref-domain} has a unique solution \(\by \mapsto u^*(\cdot, \by)\). Furthermore, \(\sinh(u^*(\cdot;\by))\) is in \(L^2(U)\) for each \(\by \in \Gamma\).
\end{prop}

\begin{proof}
	By \cref{asum4}, we have that \(\det J \geq c_2 > 0\) and so the matrix
	\begin{equation*}
		J^{-1}(\br;\by) J^{-\mathrm{T}}(\br;\by) \det J(\br;\by)
	\end{equation*}
	is symmetric positive definite for each \(\br \in U\) and \(\by \in [-1,1]^N\). The proof of weak solutions to the NPBE given in \cite[Thm.~2.14]{holst1994poisson} can easily be adjusted to the case where the matrix \(\overline{\mathbf a}:=J^{-1}(\br;\by) J^{-\mathrm{T}}(\br;\by) \det J(\br;\by)\) 
    is symmetric positive definite. That argument gives unique solutions to \cref{weak-npbe-ref-domain}. The fact that \(\sinh(u^*(\cdot;\by))\) is an \(L^2\) function also follows from the proof in \cite{holst1994poisson}.
\end{proof}

\begin{rem}
	The same assumptions also imply that there is a unique solution for \cref{weak-npbe-rand-domain} which also satisfies \(\sinh(u(\cdot;\by)) \in L^2(\mcD(\by))\) for each \(\by \in \Gamma\). 
\end{rem}

    \begin{rem}
        The next result follows from the analytic version of the Implicit Function Theorem on Banach spaces. The statement of this theorem is analogous to the finite-dimensional version, and an exact statement of it can be found in \cite{whittlesey1965analytic}. One key change from the finite-dimensional version to the infinite-dimensional version is that we now use Fr\'{e}chet derivatives and require that the derivative is an isomorphism between Banach spaces.
    \end{rem}

\begin{theo}\label{existence-of-region}
	Suppose \cref{asum1,asum2,asum3,asum4} hold. Then there exists a unique solution to \cref{strong-npbe-ref-domain}. Furthermore, there exists a complex neighborhood of \(\Gamma\) given by \(\mathcal N \subset \bbC^N\) such that there is a function
	\begin{equation*}
		\by \mapsto u^*(\cdot, \by) \in \mcH(U), \quad \text{for } \by \in \mathcal N
	\end{equation*}
	where
	\begin{enumerate}[label = (\roman*)]
		\item the map is analytic from \(\mathcal N\) into \(\mcH(U)\), and 
		\item the map agrees with the strong solution on \(\Gamma\).
	\end{enumerate}
\end{theo}

\begin{proof}
	From \cref{unique-weak-solns}, there are unique weak solutions \(u^*(\cdot,\by) \in H^1(U)\) for each \(\by \in [-1,1]^N\) and that \(\sinh(u^*(\cdot;\by))\in L^2(U)\). Setting
	\begin{equation*}
		\tilde f(\br;\by) = f^*(\br;\by) \det J(\br;\by)- (\overline{\kappa}^2)^*(\br;\by) \sinh(\tilde u^*(\br;\by) +  w^*(\br;\by))  \det J(\br;\by),
	\end{equation*}
	we get that \(\tilde f(\cdot;\by) \in L^2(U)\) and \cref{weak-eqn-ref-dom} can be written as 
	\begin{multline}
		\int_{U} \epsilon^*(\br;\by)\Big(J^{-1}(\br;\by) J^{-\mathrm{T}}(\br;\by) \det J(\br;\by) \nabla_{\br}  \tilde u^*(\br;\by) \Big)\cdot \nabla_{\br} v(\br) \\
		= \int_{U} \tilde f(\br;\by) v(\br), d\br - \int_U  \epsilon^*(\br;\by) \Big(J^{-1}(\br;\by) J^{-\mathrm{T}}(\br;\by) \det J(\br;\by) \nabla_{\br}  w^*(\br;\by) \Big)\cdot\nabla_{\br} v(\br)\, d\br, \quad \forall v\in H^1_0(U),
	\end{multline}
	which is in the same form as \cref{weak-eqn}. Then applying \cref{linear-regularity-weak-soln} gives that \(u^*(\cdot;\by) \in \mcH(U)\). To show analyticity, we apply the analytic version of the implicit function theorem. We define a mapping 
	\begin{equation}\label{mcF-function}
		\mcF: \Gamma \times \mcH(U;U_1,U_2) \to L^2(U) \times H^{1/2}(I_1) \times H^{1/2}(I_2) \times H^{3/2}(I_3) =: Z
	\end{equation} such that \(\mcF(\by, u^*) = (0,0,0,0)\) if and only if \(u^*\) is a strong solution on the reference domain for that fixed \(\by\). The first component of \(\mcF\) is defined on each \(U_k\) by 
	\begin{equation*}
		[\mcL_k(\by)] (u^*_k) + (\overline \kappa^2)^*_k(\cdot;\by) \sinh(u^*_k) \det J_k(\cdot;\by) - f_k^*(\cdot;\by) \det J_k(\cdot;\by)
	\end{equation*}
	for \(k=1,2,3\). This defines \(L^2\) function on each \(U_k\) (since \(\sinh(u^*_k) \in H^2(U_k)\)) and thus can be used to define an \(L^2\) function on all of \(U\). The second and third component of \(\mcF(\by, u^*)\) is defined to be \([\mcB_k u^* ]_{I_k}\) for \(k=1\) and \(k=2\), respectively. The final component of \(\mcF(\by, u^*)\) is given by \(\gamma_3^+ u_3^* - g^*\). The second and third components define linear maps, and the first and fourth component are easily verified to be analytic. Thus \(\mcF\) is an analytic map between Banach spaces. Fix \(\by_0\in \Gamma\). Then since \(u^*(\cdot;\by)\) is a strong solution we have 
	\begin{equation}\label{implicit-equation}
		\mcF(\by_0, u^*(\cdot;\by_0)) = (0, 0, 0, 0).
	\end{equation}
	We want to apply the implicit function theorem to \cref{implicit-equation} to get \(u^*\) as an analytic function of \(\by\) in a neighborhood around \(\by_0\). To do this, we must check that the derivative of \(\mcF\) with respect to \(u^*\) at \((\by_0, u^*(\cdot;\by_0))\) is an isomorphism between \(\mcH(U)\) and \(Z\). (\textcolor{black}{To apply the Implicit Function Theorem to \cref{implicit-equation} to get \(u^*\) as an analytic function of \(\by\) in a neighborhood around \(\by_0\), it must checked that the derivative of \(\mcF\) with respect to \(u^*\) at \((\by_0, u^*(\cdot;\by_0))\) is an isomorphism between \(\mcH(U)\) and \(Z\).}) One can compute that 
	\begin{equation*}
		[D_{u^*} \mcF(\by_0, u^*(\cdot;\by_0)](v) = ([\mcL(\by_0)](v) +  (\overline \kappa^2)^*(\cdot;\by_0) \cosh(u^*(\cdot;\by_0)) \det J(\cdot;\by_0) v, [\mcB_1 v ]_{I_1} , [\mcB_2 v ]_{I_2}, \gamma_3^+ v_3),
	\end{equation*}
	where \(\mcL\) is defined locally in \cref{linear-operator}. This linear operator is of the same form of \cref{strong-linear-interface-problem}, and so \cref{existence-uniqueness-linear-weak-solns} implies that \(D_{u^*} \mcF(\by_0, u^*(\cdot;\by_0))\) is in fact an isomorphism. Therefore the map \(\by \mapsto u^*(\cdot;\by)\) is analytic in a neighborhood of \(\by_0\in\bbR^N\). This map can be extended to a complex neighborhood of \(\by_0\in \bbC^N\). Applying this argument to every point in \(\Gamma\) gives the complex neighborhood \(\mathcal N\) for which \(\mathcal N \ni \by \mapsto u^*(\cdot;\by)\) is analytic.
\end{proof}

\subsection{Estimates on the Region of Analyticity}

\Cref{existence-of-region} shows that there exists a region of analyticity for the strong solution of the NPBE. This will be used to prove convergence results relating to the quantity of interest by using
sparse grids  \cite{nobile2008a}. \textcolor{black}{In this section, a quantitative bound on the size of the region of analyticity is derived, which is applied to obtain the aforementioned convergence rates; see \Cref{errorestimates:figure1} and the discussion that follows.} However, the rate of convergence we depend on the  size of the region of analyticity.

The typical application of the implicit function theorem does not give \emph{a priori} bounds on the size of the region of analyticity. For the finite-dimensional case, an application of Rouch\'e's theorem can give estimates of this region. The results in \cite{chang2003analytic} give simple bounds on the radius of the region of analyticity. We will show a similar result holds for Banach spaces.

\begin{theo}\label{theorem-analytic-ift-region}
	Let \(\mcF : \bbC^n \times X \to Y\) be an analytic function with \(X\) and \(Y\) Banach spaces. Suppose that \(\mcF(0,0) = 0\) and \(D_x \mcF(0,0): X \to Y\) is an isomorphism. Let \(\|D_x\mcF(0,0)^{-1} \|_{\mcL(Y,X)} \leq a\) and suppose that \(\|\mcF(z,x)\|_Y \leq M\) on B where \(B = \{(z,x): |z|, \|x\|_X \leq R\}\). Also, suppose that \(D_x \mcF(z,x)^{-1}\) exists and is a bounded operator for each \((z,x)\in B\). Then the analytic function \(z \mapsto x(z)\) is defined in a region containing the ball
	\begin{equation*}
		|z| <\Theta(M,a,R;\mcF) := \frac{\Big(aMR - \sqrt{aMR^2(aM+R)}\Big)\Big(aMR + R^2 - \sqrt{aMR^2(aM+R)}\Big)}{2a^2M^2R - R\sqrt{aMR^2(aM+R)} + aM \Big(2R^2 - 3 \sqrt{aMR^2(aM+R)}\Big)}
	\end{equation*}
	where 
	\begin{equation*}
		\| x(z) \|_X < \Xi (M, a, R;\mcF) :=  \frac{aMR+ R^2 - \sqrt{a^2M^2R^2 + aMR^3}}{aM + R} \quad \text{ for } |z|<\Theta.
	\end{equation*}
\end{theo}

\begin{proof}
	
	The proof follows in the spirit of arguments from degree theory. We find \(r > 0\) such that \(\mathcal F(z,x) \neq 0\) for \(z\neq 0\) sufficiently small and \(\|x\|_X = r\). This -- along with the assumption on the inverse of \(D_x\mcF\) -- guarantees the zero from the implicit function theorem does not leave the ball, bifurcate, or vanish. Obviously, \(z \mapsto x(z) \) is continuous and so if \(\|x(z_1)\|_X > r\) for some \(z_1\) then there must be some point \(z_0\) where \(\|x(z_0)\|_X = r\) and \(\mcF(z_0, x(z_0)) = 0\), which contradicts our assumed bound. By repeatedly applying the implicit function theorem, one can show that we cannot have the function lose analyticity if the function \(x(z)\) does not leave the ball of radius \(r\). Thus we can increase \(|z|\) up to the point where the zero can leave the ball of radius \(r\) and this becomes our estimate for the region of analyticity.
	
	We first find an \(r > 0\) such that \(\|\mcF(0, x)\|_{Y} > 0\) for all \(0< \|x\|_X \leq r\). Because \(\mcF\) is analytic, we can write \(\mcF(0,x)\) as a power series centered at \((0,0)\):
	\begin{equation}\label{eq:power-series}
		\mcF(0, x) = D_x \mcF(0,0) x + \sum_{k=2}^\infty a_k(\underbrace{x, \ldots, x}_{k \text{ times}})
	\end{equation}
	where \(a_k\) are \(k\)-linear maps. Using a Cauchy estimate we have that
	\begin{equation*}
		\|a_k(x,\ldots, x)\|_Y = M \left(\frac{\|x\|_X}{R}\right)^k.
	\end{equation*}
	Then rearranging \cref{eq:power-series}, applying \(D_x\mcF(0,0)^{-1}\) to both sides, and taking norms gives
	\begin{align*}
		\| x\|_X &= \| D_x\mcF(0,0)^{-1} [\mcF(0,x)  - \sum_{k=2}^\infty a_k(x,\cdots, x)] \|_X \\
		&\leq a \left(\|\mcF(0,x) \|_Y + \sum_{k=2}^\infty M \left(\frac{\|x\|_X}{R} \right)^k \right) \\
		&= a \left(\|\mcF(0,x) \|_Y + \frac{M\|x\|_X^2}{R^2 - \|x\|_X R}\right).
	\end{align*}
	Thus we have that 
	\begin{equation*}
		\|F(0, x) \|_Y \geq \frac{\|x\|_X} a -  \frac{M\|x\|_X^2}{R^2 - \|x\|_X R}.
	\end{equation*}
	To guarantee the right-hand side is strictly greater than zero we need that 
	\begin{equation*}
		0 < \|x\|_{X} < \frac{R^2}{R+ aM}.
	\end{equation*}
	Thus we can choose any \(r\) such that 
	\begin{equation*}
		0 < r  < \frac{R^2}{R+ aM}.
	\end{equation*}
	
	Now we want to find \(\theta> 0\) such that if \(|z| < \theta\) then \(\mcF(z, x) \neq 0\) when \(\|x\|_X = r\). It will be sufficient to find a \(\theta\) where for any \(|z| < \theta\) 
	\begin{equation}\label{rouche-ineq}
		\| \mcF(0,x) - \mcF(z, x)\|_{Y} < \|\mcF(0,x)\|_Y, \quad \text{ for } x \text{ where } \|x \| = r.
	\end{equation}
	By using a power series expansion around \((0,x)\) with respect to \(z\) and the Cauchy estimate we get that
	\begin{equation*}
		\| \mcF(0,x) - \mcF(z, x)\|_{Y} \leq \frac{M|z|}{R - |z|}.
	\end{equation*}
	Then \cref{rouche-ineq} holds if
	\begin{equation*}
		\frac{M|z|}{R - |z|} < \frac r a - \frac{Mr^2}{R^2 - rR},
	\end{equation*}
	which holds if 
	\begin{equation*}
		|z| < \frac{r R^3 - r^2R^2 - Mr^2 a R}{aMR^2+rR^2 - aMR - r^2R - Mr^2 a}.
	\end{equation*}
	Setting \(\theta\) equal to the right-hand side gives the desired result. Furthermore, the value on the right-hand side is maximized for fixed values of \(a\), \(M\), and \(R\) when
	\begin{equation*}
		r = \Xi (M, a, R;\mcF) :=  \frac{aMR+ R^2 - \sqrt{a^2M^2R^2 + aMR^3}}{aM + R}.
	\end{equation*}
	Thus the optimal radius can be given by plugging in this value for \(r\), from which we get
	\begin{equation*}
		\Theta(M,a,R;\mcF) := \frac{\Big(aMR - \sqrt{aMR^2(aM+R)}\Big)\Big(aMR + R^2 - \sqrt{aMR^2(aM+R)}\Big)}{2a^2M^2R - R\sqrt{aMR^2(aM+R)} + aM \Big(2R^2 - 3 \sqrt{aMR^2(aM+R)}\Big)}.
	\end{equation*}
\end{proof}

For the purposes of showing sub-exponential convergence of the sparse grid, we want to find the largest polyellipse in the region of analyticity. For one dimension, a Bernstein ellipse \(\mcE_{\sigma}\) is given by
\begin{equation*}
	\mcE_\sigma = \left\{z\in \bbC : \Real{z} = \frac{e^{\hat\sigma} + e^{-{\hat\sigma}}}{2} \cos(\theta) , \Imag{z} = \frac{e^{\hat\sigma} - e^{-{\hat\sigma}}}{2} \sin(\theta), \theta\in[0,2\pi), 0\leq \hat{\sigma} \leq \sigma \right\}.
\end{equation*}
For multiple dimensions, we define a polyellipse to be a direct product of Bernstein ellipses:
\begin{equation*}
	\mcE_{\sigma_1 ,\sigma_2, \ldots, \sigma_n} := \prod_{k=1}^n \mcE_{\sigma_k} \subset \bbC^n.
\end{equation*}
Applying \cref{theorem-analytic-ift-region} directly using uniform estimates for each point in \(\Gamma\) gives the analytic domain
\begin{equation*}
	\mathcal{G}_\Theta := \bigcup_{\by\in\Gamma} B_\Theta(\by),
\end{equation*}
where \(B_\Theta(\by)\) is a ball of radius \(\Theta = \Theta(M, a, R; \mcF)\) centered at \(\by\). Thus we want to fit the largest Bernstein ellipse into \(\mathcal G_\Theta\), as shown in \cref{errorestimates:figure1}.

\begin{figure}[ht]
	\begin{center}
		\begin{tikzpicture}[scale = 2, every node/.style={scale=2}]] 
			
			\node[shape=semicircle,rotate=270,fill=gray,semitransparent,inner
			sep=12.7pt, anchor=south, outer sep=0pt] at (1,0) (char) {};
			\node[shape=semicircle,rotate=90,fill=gray,semitransparent,inner
			sep=12.7pt, anchor=south, outer sep=0pt] at (-1,0) (char) {};
			\path [draw=none,fill=gray,semitransparent] (-1.001,-1) rectangle
			(1.001,1.001);
			\filldraw[fill=lightsteelblue]
			(0,0) ellipse (2 and 1);
			\node [below right,black] at
			(1.50,1.25) {$\mathcal{G}_\Theta$}; 
			\node [below
			right,black] at (0,-0.15) {${\cal E}_{\sigma_1,\dots,\sigma_{n}}$};
			\draw (1,-3pt) -- (1,3pt)   node [above] {$1$};
			\draw (-1,-3pt) -- (-1,3pt) node [above] {$-1$};
			\draw [-stealth,line width = 0.05cm] (-2.5,0) -- (2.5,0) node [above left]  {$\mathbb{R} $};
			\draw [-stealth, line width = 0.05cm] (0,-1.5) -- (0,1.5) node [below right] {$i \mathbb{R}$};
		\end{tikzpicture}
	\end{center}
	\caption{Embedding of the Bernstein polyellipse ${\cal E}_{\sigma_1,\dots,\sigma_{n}}:=\Pi^{n}_{k = 1} {\cal E}_{\sigma_k}$ in $\mathcal{G}_\Theta$.}
	\label{errorestimates:figure1}
\end{figure}
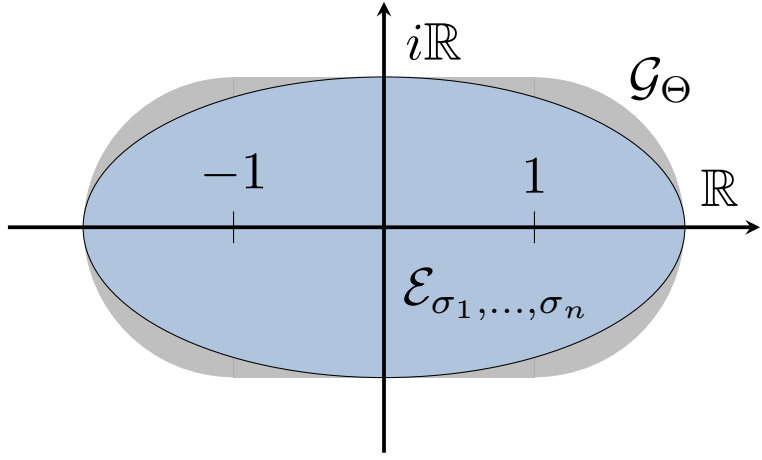

The following is a simple result following from \cref{theorem-analytic-ift-region}.
\begin{cor}\label{cor:implicit}
	Take \(\mcF\) to be the function designated in \cref{mcF-function}. Set the positive constants \(R\), \(M\), and \(a\) such that the following hold:
	\begin{enumerate}[label = (\roman*)]
		\item \(R> 0\) is small enough so that \([D_{u^*} \mcF(\by, u^*)]^{-1}\) exists whenever \(\dist(\by, \Gamma) \leq R\) and \(\|\Imag u^* \|_{\mcH(U)} \leq R\).
		\item \(M>0\) is large enough so that \(\| \mcF(\by^0 + \by, u^*(\cdot,\by^0) + u^*) \|_Z \leq M \) whenever \(\by^0 \in \Gamma\) and \(|\by|, \|u^*\|_{\mcH(U)} \leq R\).
		\item \(a>0\) is large enough so that \(\|[D_{u^*} \mcF(\by^0, u^*(\cdot;\by^0))]^{-1}\|_{\mcL(Z, \mcH(U))} \leq a\) for all \(\by^0\in \Gamma\).
	\end{enumerate}
	Then defining 
	\begin{equation}\label{sigma-star}
		\sigma_* := \log \left(\sqrt{\Theta^2 + 1} + \Theta \right)
	\end{equation}
	where \(\Theta = \Theta(M,a,R;\mcF)\), we have that the polyellipse \(\mcE_{\sigma_1, \sigma_2,\ldots,\sigma_N}\) is inside the region of analyticity for the solution \(\by \mapsto u^*(\cdot;\by)\) if \(\sigma_1 = \sigma_2 =\cdots = \sigma_N = \sigma_*\).
\end{cor}
\begin{proof}
	By applying \cref{theorem-analytic-ift-region} to each point \(\by^0 \in \Gamma\), we get that there is a region of analyticity for the solution \(\by \mapsto u^*(\cdot;\by)\), where a ball of radius \(\Theta\) centered at any \(\by^0\in \Gamma\) is contained in the region. The largest polyellipse \(\mcE_{\sigma_1, \sigma_2,\ldots, \sigma_N}\) with \(\sigma_1 = \sigma_2 = \cdots = \sigma_N\) in the region of analyticity can be computed. From \cite{castrillon2016analytic,castrillon2021stochastic}, we know the largest polyellipse occurs when \(\sigma^*\) is defined as in \cref{sigma-star}.
\end{proof}
\begin{rem}
	Conditions (ii) and (iii) of \cref{cor:implicit} are straightforwardly made to satisfy the conditions of \cref{theorem-analytic-ift-region}. The condition (i) follows from the specific form of the PDE in question. A sufficient condition for the inverse \([D_{u^*} \mcF(\by, u^*)]^{-1}\) to exist is \(\epsilon^*(\cdot;\by) > 0\), \(\det J(\cdot;\by) > 0\), \(\overline{\kappa}^2(\cdot;\by) \geq 0\), and \(\Real\cosh(u^*)\geq 0\). The first three inequalities can be satisified by choosing \(\by\) to be sufficiently close to \(\Gamma\).  The term \(\Real\cosh(u^*)\) will be strictly positive when \(u^*\) is real-valued, and only becomes negative when \(\Imag u^*\) is sufficiently large.
\end{rem}
\begin{rem}
	The estimate for the size of the polyellipse takes each \(\sigma_k\) for \(k=1,2,\ldots, N\) to be equal to \(\sigma^*\). This will only give the optimal estimate of the decay rate when using an isotropic sparse grid. For anisotropic sparse grids, we would need to choose different values for each \(\sigma_k\).
\end{rem}

\Cref{theorem-analytic-ift-region} show how we can obtain \emph{a priori} bounds on the region of analyticity after applying the implicit function theorem. To apply these bounds, one needs to find the values for the constants \(a\) and \(M\) (for a fixed \(R\)). For our problem, there must be some choice of constants that work, but computing them is tricky. The constant \(a\) can be difficult to estimate because it involves getting bounds on the solution of a \emph{backward} problem. That is, we ultimately want to find a bound on the norm of \([D_x \mcF(0, 0)]^{-1}\), which typically means solving a linear PDE. For simple linear operators and domains -- for example, a Helmholtz operator \(-\Delta  + k^2\) on the sphere-- this norm is possible to calculate explicitly. However, our domain has interfaces, which makes estimation difficult. For the moment, we set aside the problem of bounding the constant \(a\) and leave the task of optimizing the bounds for that value to future work. The constant \(M\) can be more easily estimated since we are now solving a \emph{forward} problem. That is, given some inputs to our (known) function \(\mcF\) we want to determine the size of the outputs. The remainder of this section is devoted to showing how the estimate for \(M\) can be obtained.

To get the explicit bounds needed to apply \cref{cor:implicit}, we will need to make assumptions on the parameters in the NPBE. Suppose that \(\Gamma = [-1,1]^N\). To simplify the arguments, we will assume that \(\epsilon^*_k\) and \((\overline{\kappa}^2)^*\) are piecewise constant and that \(g^* = 0\). Suppose that \(F\) has the form
\begin{equation*}
	F(\br;\by) = \br + \sum_{k=1}^N \sqrt{\mu_k} b_k(\br) y_k
\end{equation*}
where \(b_k \in C^2\). Assume that the \(b_k\)'s are normalized so that
\begin{equation*}
	\| b_k \|_{L^\infty(U)} = 1 \quad \text{ for } k =1,2,\ldots N,
\end{equation*}
and that the \(\mu_k\)'s are decreasing in value. Then the Jacobian, \(J(\br;\by)\) has the form
\begin{equation*}
	J(\br;\by) = I + \sum_{k=1}^N \sqrt{\mu_k} B_k(\br) y_k
\end{equation*}
where \(B_k(\br) = \partial b_k(\br)\). Denote 
\begin{equation}\label{mcB-defn}
	\mcB \by := \sum_{k=1}^N \sqrt{\mu_k} B_k(\cdot) y_k
\end{equation}
so that \(J(\cdot;\by) = I + \mcB\by\). Note that we can treat \(\mcB\) as a linear map from \(\bbC^N\) into \(C^1(U,\bbC^{3\times 3})\).

We will also assume that \(f^*\) has the form
\begin{equation*}
	f^*(\br;\by) =   \sum_{k=1}^{N_f} \xi(F(\br;\by) - F(\eta_k;\by))
\end{equation*}
where \(\xi\) is a Gaussian function and each \(\eta_k \in U\) is a fixed point.

Suppose that \(\mcF(\by^0,u^0) = 0\) for some \(\by^0 \in \Gamma\) and \(u^0 \in \mcH(U)\). To apply \cref{cor:implicit}, we want an estimate on 
\begin{equation}\label{difference-estimate}
	\|\mcF(\by^0 + \by , u^0 + u) - \mcF(\by^0, u^0)\|_{Z}.
\end{equation}
Note that the norm on the last three coordinates of \(\mcF\) (see \cref{mcF-function}) can be estimated by finding bounds for the trace operators and co-normal derivatives. Let us focus on the first coordinate of \cref{difference-estimate}, which is given by 
\begin{equation*}
	\begin{aligned}
		&[\mcL_k(\by^0 + \by)] (u^0_k + u_k) + (\overline \kappa^2)^*_k \sinh(u^0_k + u_k) \det J_k(\cdot;\by^0 + \by) - f_k^*(\cdot;\by^0 + \by) \det J_k(\cdot;\by^0 + \by) \\
		&\quad - [\mcL_k(\by^0)] (u^0_k) + (\overline \kappa^2)^*_k \sinh(u^0_k) \det J_k(\cdot;\by^0) - f_k^*(\cdot;\by^0) \det J_k(\cdot;\by^0 )
	\end{aligned}
\end{equation*}
for \(k = 1,2,3\). 

Estimating the above term in the \(L^2(U)\) requires us to define the norms of several other spaces in order for the calculation to be tractable. Norms in finite-dimensional vector spaces (e.g.\ \(\bbR^n\)) will be denoted with single bars, \(|\cdot|\), while norms in infinite-dimensional function spaces will be denoted with double bars, \( \| \cdot \| \). Similar notation will be used for the norms induced on linear operators on normed spaces. In particular, \(|\cdot|_p\) for \(1\leq p \leq \infty\) will be used to denote the typical \(\ell^p\) norms in \(\bbR^n\) or \(\bbC^n\) as well as the associated matrix norms. So if \(\bv \in \bbC^n\) and \(A\in \bbC^{n\times n}\), then \(|\mathbf v|_2\) will be the standard Euclidean norm of \(\bv\) and \[|A|_2 = \sup_{\bx \in \bbC^n\setminus\{0\}} \frac{|A\bx|_2}{|\bx|_2}.\]

We will assume \(L^2(U;\bbC^n)\) to have the norm
\begin{equation*}
	\| \bv(\cdot) \|_{L^2(U;\bbC^n)} = \| \,|\bv(\cdot)|_2 \,\|_{L^2(U;\bbR)} = \left( \int_U |\bv(\br) |_2^2\, d\br  \right)^{1/2}.
\end{equation*}
We will also assume that \(H^1(U;\bbC^n)\) has the norm
\begin{equation*}
	\| \bv(\cdot) \|_{H^1(U;\bbC^n)} = \|\bv(\cdot)\|_{L^2(U;\bbC^n)} + \sum_{i=1}^n \left\| \frac{\partial \bv}{\partial r_i} (\cdot) \right \|_{L^2(U;\bbC^n)}.
\end{equation*}
For the space \(C^1(U;\bbC^{3\times 3})\), we introduce the norm
\begin{equation*}
	\| B(\cdot) \|_{C^1(U;\bbC^{3\times 3})} := \max_{\substack{k=0,1 \\ i = 1,2,3}} \sup_{\br\in U} \left|\frac{\partial^k}{\partial r_i^k}B(\br)\right|_2.
\end{equation*}
Note that for any \(B\in C^1(U;\bbC^{3\times3})\) and \(\bv \in H^1(U;\bbC^3)\), we have 
\begin{equation*}
	\| B \bv \|_{H^1(U;\bbC^3)} \leq 4 \| B \|_{C^1(U;\bbC^{3\times 3})} \| \bv \|_{H^1(U;\bbC^3)},
\end{equation*}
which implies the continuous imbedding \(C^1(U;\bbC^{3\times 3}) \hookrightarrow \mcL(H^1(U;\bbC^3))\).

Recall that \(\mcB\) as defined in \cref{mcB-defn} is a linear map from \(\bbC^N\) into \(C^1(U;\bbC^{3\times3})\), and so inherits a natural norm from being a bounded linear operator between two Banach spaces. We introduce a different norm for these maps that will be easier to estimate and be an upper bound for the linear operator norm. Let 
\begin{equation*}
	\| \mcB \|_p := \left(\sum_{k=1}^N |\mu_k|^{p/2} \| B_k \|_{C^1(U,\bbC^{3\times 3})}^p  \right)^{1/p}
\end{equation*}
for \(1\leq p < \infty\) and 
\begin{equation*}
	\| \mcB \|_\infty := \max_{k=1,2,\ldots,N} \sqrt{\mu_k} \| B_k \|_{C^1(U, \bbC^{3\times 3})}.
\end{equation*}
Then for \(1\leq p , q \leq \infty\) with \(\frac 1 p + \frac 1 q = 1\), we have that 
\begin{equation}\label{holder}
	\| \mcB \by \|_{C^1(U,\bbC^{3\times 3})} \leq \|\mcB \|_{p} |\by|_q.
\end{equation}

For the following estimates on various parameters, the hypothesis 
\begin{equation}\label{small-b}
	\| \mcB \|_1 < \frac 1 4.
\end{equation}
is assumed, which defines  \(J^{-1}(\by^0+\by)\) in \(\mcL(H^1(U;\bbC^3))\).
\begin{prop}\label{prop:computations}
	Let \(\mcL = \mcL(H^1(U;\bbC^3))\) denote the space of bounded linear operators from \(H^1(U;\bbC^3)\) to itself. Suppose that \(\by^0\in\Gamma\) and \[|\by|_\infty < \frac 1 {4\|\mcB\|_1} - |\by^0|_\infty.\] Then we have the following bounds:
	\begin{align}
		\| J^{-1}(\by^0) \|_\mcL &\leq \frac{1}{1 - 4 \| \mcB \|_1 |\by^0|_\infty}, \label{Jinf-L} \\
		\| J^{-1}(\by^0 + \by) \|_\mcL &\leq \frac{1}{1 - 4 \| \mcB \|_1 (|\by^0|_\infty + |\by|_\infty)}, \label{Jinf-L2} \\
		\| (I + J^{-1}(\by^0)\mcB \by)^{-1} - I\|_{\mcL} &\leq \frac{4 \| \mcB \|_1 |\by|_\infty}{1 - 4 \| \mcB \|_1 (|\by^0|_\infty + |\by|_\infty)}, \label{I-J-B} \\
		|\det J(\by^0)| &\leq \frac 1 {(1 - \|\mcB\|_1 |\by^0|_\infty)^3},\label{absdet} \\
		|\det J(\by^0 + \by)| &\leq \frac 1 {(1 - \|\mcB\|_1 (|\by^0|_\infty + |\by|_\infty))^3}, \label{absdet2}\\
		\| \det J(\by^0) \|_{\mcL} & \leq \frac 4 {(1 - \|\mcB\|_1 |\by^0|_\infty)^3} \label{det-L} \\
		\| \det J(\by^0 + \by)  \|_{\mcL} & \leq \frac 4 {(1 - \|\mcB\|_1 (|\by^0|_\infty + |\by|_\infty))^3} \label{det-L2}\\
		\| \det(I + J^{-1}(\by^0)\mcB\by) -1 \|_\mcL &\leq 4 \left[ \left(\frac{1 - \|\mcB\|_1 |\by^0|_\infty }{1 - \|\mcB\|_1(|\by^0|_\infty + |\by|_\infty)} \right)^3  - 1\right] \label{det-I-J-B}.
	\end{align}
\end{prop}
\begin{proof}
See \cref{app:computations} for the proof.
\end{proof}
Hence we have that 
\begin{equation}\label{estimate-linear}
	\begin{aligned}
		\|&[\mcL_k(\by^0 + \by)] (u^0_k + u_k) - [\mcL_k(\by^0)] (u^0_k)\|_{L^2(U_k)} \\
		&\leq |\epsilon_k^*|\, \| J^{-1}(\by^0 + \by)J(\by^0+\by)^{-\mathrm{T}} \det J(\by^0+\by) \nabla(u^0_k + u_k)  \\
		&\hspace{20em} - J^{-1}(\by^0)J(\by^0)^{-\mathrm{T}} \det J(\by^0) \nabla u^0_k\|_{H^1(U_k;\bbC^3)} \\
		&\leq |\epsilon_k^*| \Big(a(\by^0, \by) \| u_k \|_{H^2(U_k)} + b(\by^0, \by) \|u^0_k\|_{H^2(U_k)} \Big)
	\end{aligned}
\end{equation}
where 
\begin{equation*}
	a(\by^0, \by) := \left(\frac{1}{1 - 4 \| \mcB \|_1 (|\by^0|_\infty + |\by|_\infty)}\right)^2 \times \frac 4 {(1 - \|\mcB\|_1 (|\by^0|_\infty + |\by|_\infty))^3}
\end{equation*}
and 
\begin{equation*}
	\begin{aligned}
		b(\by^0,\by) := &\Bigg[ \left(\frac{4 \| \mcB \|_1 |\by|_\infty}{1 - 4 \| \mcB \|_1 (|\by^0|_\infty + |\by|_\infty)}\right)^2 \times 4 \left( \left(\frac{1 - \|\mcB\|_1 |\by^0|_\infty }{1 - \|\mcB\|_1(|\by^0|_\infty + |\by|_\infty)} \right)^3  - 1\right) \\
		&\quad+2 \times \frac{4 \| \mcB \|_1 |\by|_\infty}{1 - 4 \| \mcB \|_1 (|\by^0|_\infty + |\by|_\infty)} \times 4 \left( \left(\frac{1 - \|\mcB\|_1 |\by^0|_\infty }{1 - \|\mcB\|_1(|\by^0|_\infty + |\by|_\infty)} \right)^3  - 1\right) \\
		&\quad+ \left(\frac{4 \| \mcB \|_1 |\by|_\infty}{1 - 4 \| \mcB \|_1 (|\by^0|_\infty + |\by|_\infty)}\right)^2 \\
		&\quad + 2 \times \frac{4 \| \mcB \|_1 |\by|_\infty}{1 - 4 \| \mcB \|_1 (|\by^0|_\infty + |\by|_\infty)} \\
		&\quad + 4 \left( \left(\frac{1 - \|\mcB\|_1 |\by^0|_\infty }{1 - \|\mcB\|_1(|\by^0|_\infty + |\by|_\infty)} \right)^3  - 1\right) \Bigg] \times \left( \frac{1}{1 - 4 \| \mcB \|_1 |\by^0|_\infty}  \right)^2 \times \frac 4 {(1 - \|\mcB\|_1 |\by^0|_\infty)^3}.
	\end{aligned}
\end{equation*}
It can then be shown that the norm in \(L^2(U)\) is bounded by
\begin{equation*}
	\sqrt 3 \, \epsilon_{\max}\left( a(\by^0,\by) \| u \|_{\mcH}  + b(\by^0, \by) \|u^0\|_\mcH \right),
\end{equation*}
where \(\epsilon_{\max} := \max\{\epsilon_1, \epsilon_2,\epsilon_3\}\).

Let \(C_k>0\) denote the constant associated with the Banach algebra \(H^2(U_k)\). That is, \(C_k>0\) is a constant such that
\begin{equation*}
	\|uv\|_{H^2(U_k)}\leq C_k \| u \|_{H^2(U_k)} \times \| v \|_{H^2(U_k)}\quad \forall u,v\in H^2(U_k).
\end{equation*}

Then we have 
\begin{equation*}
	\begin{aligned}
		&\|\overline{\kappa}_k^2 \sinh(u^0_k + u_k) \det J(\by^0 + \by) - \overline{\kappa}_k^2 \sinh(u^0_k) \det J(\by^0) \|_{L^2(U_k)}   \\
		&\quad\leq \frac {2 |\overline{\kappa}^2_k|  \cosh(C_k(\|u^0_k\|_{H^2(U_k)} + \frac 1 2 \|u_k\|_{H^2(U_k)} )) \sinh(C_k\frac 1 2 \|u_k\|_{H^2(U_k)})} {(1 - \|\mcB\|_1 |\by^0|_\infty)^3} \left(\frac{1 - \|\mcB\|_1 |\by^0|_\infty }{1 - \|\mcB\|_1(|\by^0|_\infty + |\by|_\infty)} \right)^3 \\
		&\qquad + \frac { |\overline{\kappa}^2_k|  C_k^{-1} \sinh(C_k \|u_k\|_{H^2(U_k)})} {(1 - \|\mcB\|_1 |\by^0|_\infty)^3} \left[ \left(\frac{1 - \|\mcB\|_1 |\by^0|_\infty }{1 - \|\mcB\|_1(|\by^0|_\infty + |\by|_\infty)} \right)^3  - 1\right].
	\end{aligned}
\end{equation*}
Thus the \(L^2(U)\) norm is bounded by 
\begin{multline}\label{estimate-nonlinear}
	\frac{\sqrt 3\, \overline{\kappa}^2_{\max}}{(1 - \| \mcB \|_1 |\by^0|_\infty)^3} \Bigg[ 2 \cosh(C_{\max}(\|u^0_k\|_{\mcH} + \frac 1 2 \|u_k\|_{\mcH} )) \sinh(C_{\max}\frac 1 2 \|u_k\|_{\mcH}) \left(\frac{1 - \|\mcB\|_1 |\by^0|_\infty }{1 - \|\mcB\|_1(|\by^0|_\infty + |\by|_\infty)} \right)^3  \\ 
	+   C_{\max}^{-1} \sinh(C_{\max} \|u_k\|_{\mcH})  \left[ \left(\frac{1 - \|\mcB\|_1 |\by^0|_\infty }{1 - \|\mcB\|_1(|\by^0|_\infty + |\by|_\infty)} \right)^3  - 1\right] \Bigg]
\end{multline}
where \(C_{\max} :=\max\{C_1, C_2, C_3\}\).

Finally the forcing term can be dealt with by finding the derivatives with respect to \(y_k\) for \(k=1,2,\ldots, N\). We have that
\begin{equation*}
	\begin{aligned}
		&\frac{\partial}{\partial y_i} \xi(F(r;\by) - F(\eta_k;\by)) \det J(\br;\by) \\
		&\quad = \sum_{j=1}^3 \frac{\partial \xi}{\partial x_j} (F(\br; \by) - F(\eta_k; \by)) [\sqrt{\mu_i} (b^j_i(\br) - b^j_i(\mathbf{\eta}_k))] \det J(\br;\by) +  \xi(F(r;\by) - F(\eta_k;\by)) \\
		&\qquad + \xi(F(r;\by) - F(\eta_k;\by)) \det J(\br;\by) \mathrm{tr}\left(J^{-1}(\br;\by) \sqrt{\mu_i} B_i(\br)\right)
	\end{aligned}
\end{equation*}
and so
\begin{equation*}
	\left \| \frac{\partial}{\partial y_i} \xi(F(\cdot;\by) - F(\eta_k;\by)) \det J(\cdot;\by)\right\|_{L^2(U)} \leq 6 \sqrt{\mu_i} \max_{j=1,2,3} \left \| \frac{\partial \xi}{\partial x_j} \right\|_{L^2(\bbR^3)} + \|\xi\|_{L^2(\bbR^3)} \times \frac{3}{1-\|\mcB\|_1 |\by|} \times \|\mcB\|_\infty
\end{equation*}
Thus we get that
\begin{equation}\label{estimate-f}
	\begin{aligned}
		&\| f^*(\cdot;\by^0 + \by) \det J(\cdot;\by^0 + \by)  -  f^*(\cdot;\by^0 ) \det J(\cdot;\by^0 ) \|_{L^2(U)}  \\
		&\quad\leq N_f \cdot |\by| \cdot \left( 6 \sqrt{\mu_i} \max_{j=1,2,3} \left \| \frac{\partial \xi}{\partial x_j} \right\|_{L^2(\bbR^3)} +  \frac{3 \|\xi\|_{L^2(\bbR^3)} \|\mcB\|_\infty}{1-\|\mcB\|_1 |\by|}  \right).
	\end{aligned}
\end{equation}

Combining the estimates in \cref{estimate-linear,estimate-nonlinear,estimate-f} gives a bound on the \(L^2(U)\) part of 
\(\mcF(\by^0 + \by , u^0 + u) - \mcF(\by^0, u^0)\). We still have constants that have not been explicitly given (such as \(C_{\max}\) and the norm of the solution \(u^0\)), but estimating these would be difficult to do in the space of this paper.

\section{Sparse Grids}\label{sparsegrids}

Sparse grids are a mathematical technique used to efficiently approximate functions and solve problems in high-dimensional spaces. They provide a way to reduce the computational cost associated with high-dimensional problems by exploiting the sparsity of the underlying function.
In many real-world applications, such as optimization, machine learning, and scientific simulations, the dimensionality of the problem can be quite large. Traditional numerical methods often struggle to handle these high-dimensional scenarios due to the exponential growth of computational requirements with increasing dimensions.
Sparse grids offer a solution to this problem by selectively evaluating the function only at a subset of points in the high-dimensional space. The idea is to concentrate computational effort on regions that contribute the most to the overall approximation accuracy, while ignoring or approximating the function in less significant regions.

Sparse grids are constructed from tensor products of Lagrange iterpolation. Given a set of data points $(\zeta_0, z_0), (\zeta_1, z_1), \ldots, (\zeta_p, z_p) \in \tilde \Gamma \times \bbR $, where we define \(\tilde \Gamma := [-1,1]\) and the $\zeta_i$ values are distinct, Lagrange interpolation constructs a polynomial $P(\zeta)$ of degree at most $n$ that satisfies:
\[
	P(\zeta_i) = z_i, \quad \text{for } i = 0, 1, \ldots, p
\]
The polynomial $P(\zeta)$ is defined as the linear combination of Lagrange basis polynomials $l_i(\zeta)$, which are constructed to ensure that $P(\zeta_i) = z_i$ for each data point:
\[
	P(\zeta) = \sum_{i=0}^{p} z_i l_i(\zeta)
\]
The Lagrange basis polynomials are defined as:
\[
	l_i(\zeta) = \prod_{\substack{j=0 \\ j \neq i}}^{n} \frac{\zeta - \zeta_j}{\zeta_i - \zeta_j}
\]

These basis polynomials have the property that $l_i(\zeta_i) = 1$ and $l_i(\zeta_j) = 0$ for $j \neq i$, ensuring that the polynomial $P(\zeta)$ passes through the corresponding data point $(\zeta_i, z_i)$. It is clear that $P(\zeta) \in \mcP_{p}(\tilde \Gamma) := \text{\rm span}\{\zeta^m : \,m=0,\dots,p\}$.

We can now extend Lagrange interpolation to using  tensor products of 1D interpolants. Let
\[
	C^{0}(\Gamma) : = \{ v: \Gamma \rightarrow V\,\, \mbox{is continuous on $\Gamma$ and } \max_{y\in \Gamma} |v(y)| < \infty \},
\]
$\pp = (p_1,\ldots,$ $p_{N})$, and $\mcP_{ p_n}(\tilde \Gamma):=\text{\rm span}(y_n^m,\,m=0,\dots,p_n),$  for each dimension $n = 1,\dots,N$. Let
$\mcI^{m(i)}:C^{0}(\tilde \Gamma) \rightarrow \mcP_{m(i)-1}(\tilde \Gamma)$ be the Lagrange interpolation operator  where $i \in \bbN_0$, $m(0) = 0$, $m(1) = 1$ and in general $m(i) \in \bbN_0$ is the number of evaluation points at level $i$. Note that if $m(i) = 0$ then let $\mcP_{-1}(\Gamma) := \emptyset$.

Consider the vector of approximations $\ii = (i_1,i_2, \dots,i_N) \in \bbN^{N}_{0}$, and form the  space $\mcP_{\pp}(\Gamma) = \bigotimes_{n=1}^{N}\;\mcP_{p_n}(\tilde \Gamma)$ then the Lagrange interpolation for $N$ dimensions operator $\mcI^{N}_{\ii} :C^{0}(\Gamma) \rightarrow \mcP_{\pp}(\Gamma)$ can now be built as
\[
	\mcI^{N}_{\ii} = \mcI^{m(i_1)}_{1} \otimes \mcI^{m(i_2)}_{2}  \otimes \dots \otimes \mcI^{m(i_N)}_{N}.
\]
More explicitly for each dimension $n = 1, \dots, N$ let $\{y^{n}_{1},\dots,y^{n}_{m(i)}\} \subset \tilde \Gamma$ be a sequence of abcissas for the Lagrange interpolation operator $\mcI^{m(i_n)}_{n}$. Thus for any $\nu \in C^{0}(\Gamma)$
\[
	\mcI^{N}_{\ii} \nu(\by) = \sum_{k_1 = 1}^{m(i_1)} \sum_{k_2 = 1}^{m(i_2)} \dots \sum_{k_N = 1}^{m(i_N)}  \nu(y^{i_{1}}_{k_1},y^{i_{2}}_{k_2},\dots,y^{i_{N}}_{k_N}) l^{i_1}_{k_1} \otimes l^{i_2}_{k_2} \otimes \dots \otimes l^{i_N}_{k_N},
\]
where 
\[
	l^i_j(y) = \prod_{\substack{k=0 \\ j \neq k}}^{N} \frac{y - y^i_k}{y^i_k - y^j_k}.
\]
However, the dimensionality of $\mcP_{\pp}$ explodes as $\prod_{n=1}^N$ $(p_n+1)$ making Lagrange interpolation intractable for even a number of moderate dimensions. In contrast, if there exists a complex analytic extension of $\nu(\by)$ with respect to $\by$ then sparse grids are a better choice \cite{Smolyak63,Novak_Ritter_00,Back2011,nobile2008a}. They provide almost the same convergence accuracy of full tensor product grids, but with significant reductions in dimensionality. This is achieved by judiciously selecting a reduced set of monomials from the full tensor product. 

Let $\mm(\ii) = (m(i_1),\ldots,m(i_{N})) \in \bbZ^{N}$ be the vector of the number of evaluations points for each dimension. For a given non-negative integer $w$, we define the index set $\Lambda^{m,g}(w)$ as follows:
\[
	\Lambda^{m,g}(w) = \{\pp\in\Nset^{N}, \;\;  g(\mm^{-1}(\pp+\oone))\leq w\}.
\]
In this context, the function $g:\bbZ^{N} \rightarrow \bbZ$ acts as a restriction function along each dimension of the complete tensor grid.

The indices in $\Lambda^{m,g}(w)$ constitute the set of permissible polynomial moments $\Pol_{\Lambda^{m,g}(w)}(\Gamma)$, subject to the restrictions imposed by $(\mm,g,w)$. Specifically, this polynomial set is defined as:
\[
	\Pol_{\Lambda^{m,g}(w)}(\Gamma) := \mathrm{span}\left\{\prod_{n=1}^{N} y_n^{p_n}, \text{ with } \pp\in\Lambda^{m,g}(w)\right\}.
\]
Let's consider the difference operator along the $n^{th}$ dimension of $\Gamma$, denoted as 
\[
	{\Delta_n^{m(i)} :=} \mcI_n^{m(i)}-\mcI_n^{m(i-1)}.
\]
By taking the tensor product of these difference operators across all dimensions, we can construct a sparse grid. In this context, $w \in \bbN_0$ represents the desired approximation level. The sparse grid approximation of $\nu$ is then obtained as follows:
\[
	\mcS^{m,g}_w[\nu] = \sum_{\ii\in\bbZ^{N}: g(\ii)\leq w} \;\; \bigotimes_{n=1}^{N} {\Delta_n^{m(i_n)}}(\nu(\by)). 
\]

We have the flexibility to choose different values for the parameters $m$ and $g$. Our main objective is to achieve accurate results while controlling the increase in dimensionality of the space $\Pol_{\Lambda^{m,g}(w)}(\Gamma)$. To address this, we can utilize the well-known Smolyak sparse grid method introduced by Nobile et al. (2008), which can be constructed using the following formulas:
\[
	m(i) = \begin{cases}
			1, & \text{for } i=1 \\ 
			2^{i-1}+1, & \text{for } i>1 
		   \end{cases}\quad \text{ and } \quad g(\ii) = \sum_{n=1}^N (i_n-1).
\]
For this particular choice, the index set $\Lambda^{m,g}(w)$ is defined as follows: $\Lambda^{m,g}(w):=\{\pp\in\Nset^{N}: \;\; \sum_n f(p_n) \leq w\}$ where
\[
	f(p) = \begin{cases}
			0, \; p=0 \\
			1, \; p=1 \\
			\lceil \log_2(p) \rceil, \; p\geq 2
		   \end{cases}.
\]
Alternative choices, such as the Total Degree (TD) and Hyperbolic Cross (HC) grids, are described in \cite{castrillon2016analytic}. 

The last component of the sparse grid is the selection of the abcissas $\{y^{n}_{1},\dots,y^{n}_{m(i)}\} \subset [-1,1]$ along each dimension. One option is the extrema of Chebyshev polynomials:
\[
	y^n_j = -\cos \left( \frac{\pi(j-1)}{m(i) - 1} \right).
\]
This popular choice are denoted as Clenshaw-Curtis abscissas. It is worth noting that not all stochastic dimensions need to be treated equally. Some dimensions may contribute more to the sparse grid approximation than others. By customizing the restriction function $g$ according to the input random variables $y_n$ for $n = 1, \dots, N$, a more accurate \emph{anisotropic} sparse grid can be obtained \cite{Schillings2013,nobile2008b}. In this paper, for the sake of simplicity, we focus on isotropic sparse grids. However, extending the approach to an anisotropic setting is a straightforward.

Our current focus is on establishing error bounds for the sparse grid, specifically the norm $\|\nu - \mcS^{m,g}[\nu]$ $\|_{{L^{\infty}(\Gamma)}}$. This bound can be controlled by three key factors. Firstly, the number of dimensions, denoted as $N$, influences the bound. Secondly, the number of knots, denoted as $\eta$, in the sparse grid plays a role. However, the most crucial factor is the size of the complex region in the analytic extension of $\nu(\by)$ onto $\bbC^{N}$ and the following bound on the polyellipse
\[
	\tilde{M}(\nu) := \sup_{\bz \in \mcE_{\sigma_1, \dots, \sigma_{N}}} |\nu(\bz)|.
\]

With the parameters analytic extension parameters $(\sigma^{*}, \tilde M(\nu))$, the number of dimensions $N$ and  the level of the sparse grid $w$ the error of the sparse grid $\|\nu - \mcS^{m,g}[\nu]$ $\|_{{L^{\infty}(\Gamma)}}$ can be bounded. Define the following constants
\begin{align*}
		\sigma &= \sigma^*/2, & \tilde{C}_{2}(\sigma) &= 1 + \frac{1}{\log{2}}\sqrt{\frac{\pi}{2\sigma}}, & \delta^{*}(\sigma) &= \frac{e\log{(2)} - 1}{\tilde{C}_2 (\sigma)},  \\
		\mu_1 &= \frac{\sigma}{1 + \log (2N)}, & \mu_2(N) &= \frac{\log(2)}{N(1 + \log(2N))}, & \mu_3 &= \frac{\sigma \delta^{*} \tilde C_2(\sigma)}{1 + 2\log(2N)},
\end{align*}
\begin{align*}
	a(\delta,\sigma) &=\exp{\left(\delta \sigma \left\{\frac{1}{\sigma \log^{2}{(2)}}+ \frac{1}{\log{(2)}\sqrt{2 \sigma}}+ 2\left( 1 + \frac{1}{\log{(2)}} \sqrt{ \frac{\pi}{2\sigma} }\right)\right\}\right)}, \\
	C_1(\sigma,\delta,\tilde M(\nu)) &= \frac{4\tilde{M}(\nu) C(\sigma)a(\delta,\sigma)}{ e\delta\sigma}, \\ 
	{\mathcal Q}(\sigma,\delta^{*}(\sigma),N, \tilde M(\nu)) &= \frac{ C_1(\sigma,\delta^{*}(\sigma),\tilde M(\nu))}{\exp(\sigma\delta^{*}(\sigma) \tilde{C}_2(\sigma) )}\frac{\max\{1,C_1(\sigma,\delta^{*}(\sigma),\tilde M(\nu))\}^{N}}{|1- C_1(\sigma,\delta^{*}(\sigma),\tilde M(\nu))|}.
\end{align*}

\begin{theo}
	Suppose that $\nu \in C^{0}(\Gamma;\bbR)$ admits an analytic extension on ${\mathcal E}_{\sigma_1, \dots, \sigma_{N}}$.   Let $\mcS^{m,g}_{w}[\nu]$ be the sparse grid approximation of the function $\nu$ with Clenshaw-Curtis abcissas. If $w > N / \log{2}$ then 
	\begin{equation*}   
		\|\nu - \mcS^{m,g}_{w}[\nu]
		\|_{L^{\infty}(\Gamma)} 
		\leq 
		{\mathcal
			Q}(\sigma,\delta^{*}(\sigma),N, \tilde M(\nu))
		\eta^{\mu_3(\sigma,\delta^{*}(\sigma),N)}\exp
		\left(-\frac{N \sigma}{2^{1/N}} \eta^{\mu_2(N)} \right)
		, \\
		\label{erroranalysis:sparsegrid:estimate}
	\end{equation*}
	Furthermore, if $w \leq N / \log{2}$ then the following algebraic
	convergence bound holds:
	\begin{equation*}
		\begin{split}
			\| \nu - \mcS^{m,g}_{w}[\nu] \|_{L^{\infty}(\Gamma)}
			&\leq 
			\frac{C_1(\sigma,\delta^{*}(\sigma),\tilde M(\nu))
				\max{\{1,C_{1}(\sigma,\delta^{*}(\sigma),\tilde{M}(\nu))
					\}}^N
			}{
				|1 - C_1(\sigma,\delta^{*}(\sigma),\tilde M(\nu))|
			}\eta^{-\mu_1}.
		\end{split}
		\label{erroranalysis:sparsegrid:estimate2}
	\end{equation*}
	\label{erroranalysis:theorem1}
\end{theo}
\begin{proof} Theorem 3.10 and 3.11 in \cite{nobile2008a}.
\end{proof}

\section{Numerical Results}\label{numerics}

We now test the complex analyticity result for the NPBE by computing the potential field of the Trypsin protein (PDB:1ppe \cite{Berman2000}, $n = 1{,}852$ atoms) submerged in a solvent. In Figure \ref{NumericalResults:Fig1} (a) the secondary structure of the Trypsin molecule is rendered
with a meshed surface of the molecular boundary. This corresponds to the molecular boundary obtained by rolling a solvent atom around the molecule. This boundary corresponds to the interface $I_1$. Inside the molecule the dielectric is set to $\epsilon = 70$ and outside the boundary it is set to $\epsilon = 1$ (e.g. solvent dielectric). Note that these dielectric values are unit-less. The second boundary $I_2$ (not rendered) corresponds to the ion-accessible surface. The Debye-H\"uckel parameter, $\overline{\kappa}$, is set to zero inside the surface and non-zero outside. The entire protein is contained in a cubic domain $\mcD$ measuring $70 \times 70 \times 70$ \AA\,  and the Dirichlet boundary conditions are set to zero, i.e., $u \equiv 0$ on \(\partial \mcD\). The temperature of the solvent is set to $T = 310$ Kelvin. Let $\mcC := \{ \bx_1, \dots, \bx_n\}$ correspond to set of the location of the molecular atoms. This information is contained in the PDB file. In theory and in practice these locations represent point charges that are replaced with functions $L^{2}(\Omega)$ as this guarantee the existence of a unique solution for the NPBE \cite{holst1994poisson}. The potential field, which corresponds to the solution of the NPBE, is then solved using APBS.

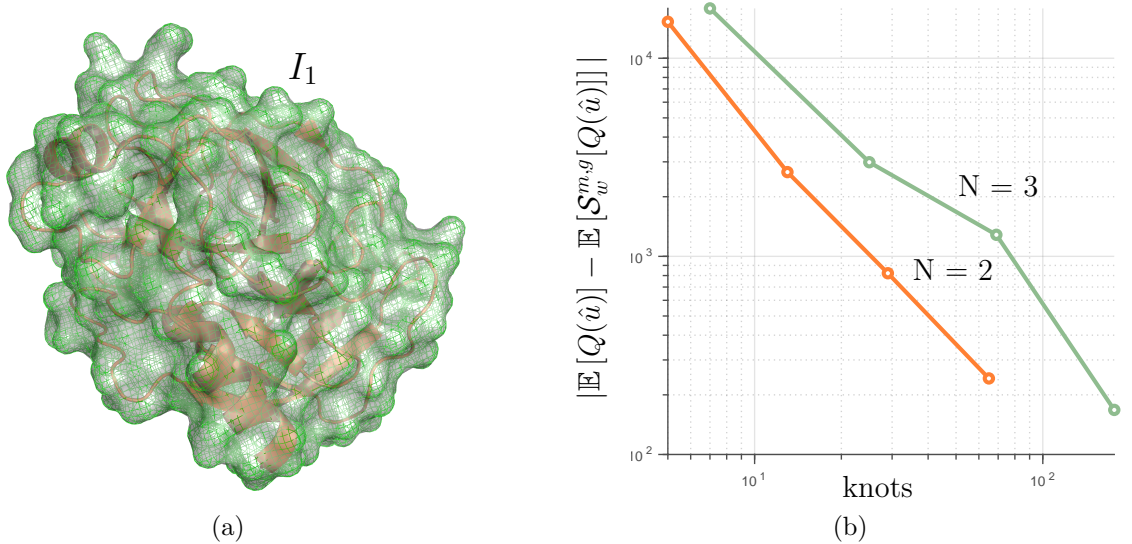
\begin{figure}[h]
	\centering
	\begin{tabular}{c c}
		\begin{tikzpicture}
			\node at (0,0) {\includegraphics[scale=0.4, trim = 5cm 0cm 5cm 0cm, clip]{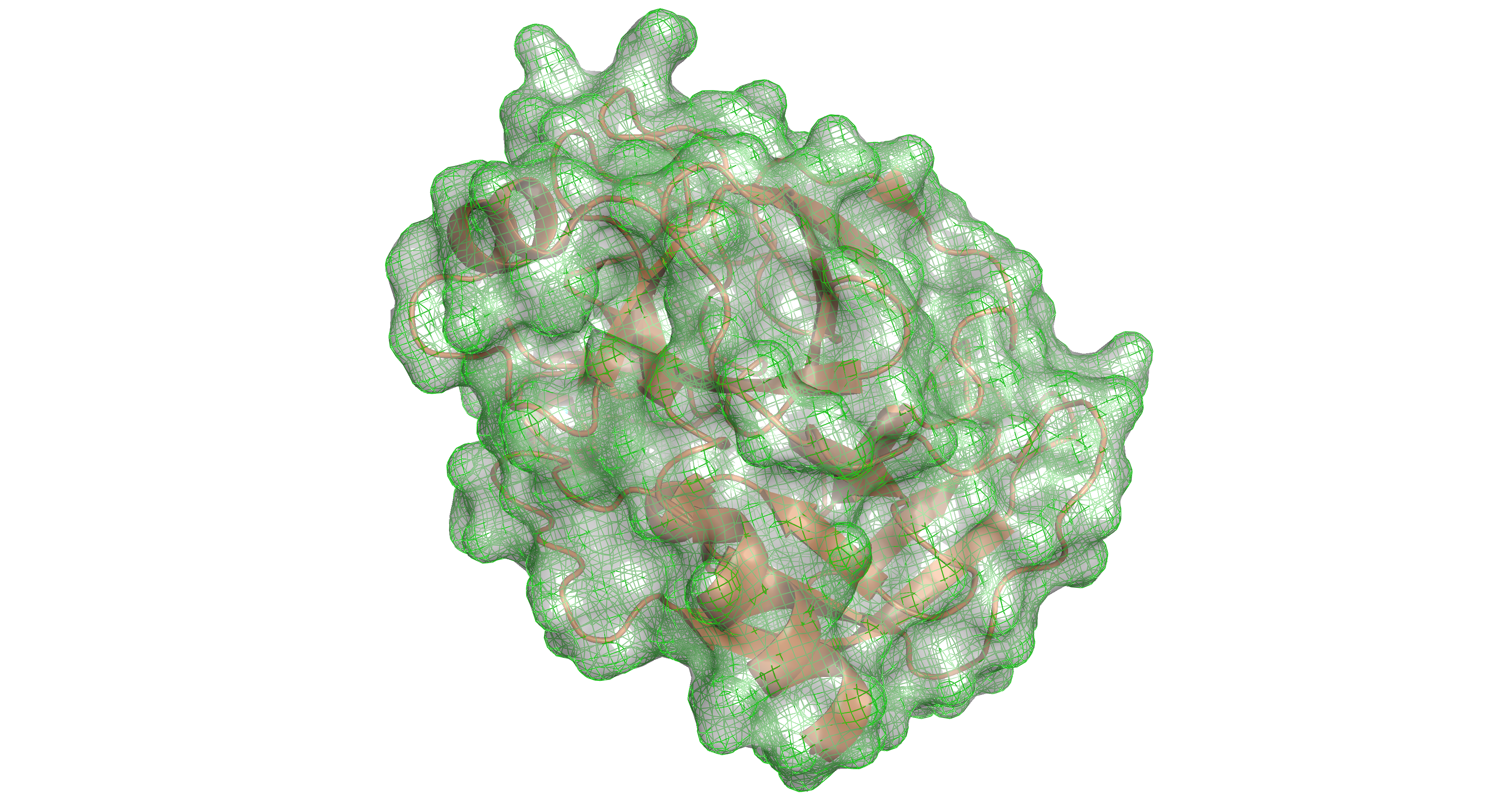}};
			\node at (1,2.5) {\Large $I_1$};
		\end{tikzpicture}
		&
		\begin{tikzpicture}[scale = 0.8, every node/.style={scale=0.8}]]   
			\node at (0,0) {\includegraphics[scale=0.65, trim = 4.5cm 7.5cm 4.5cm 7.75cm, clip]{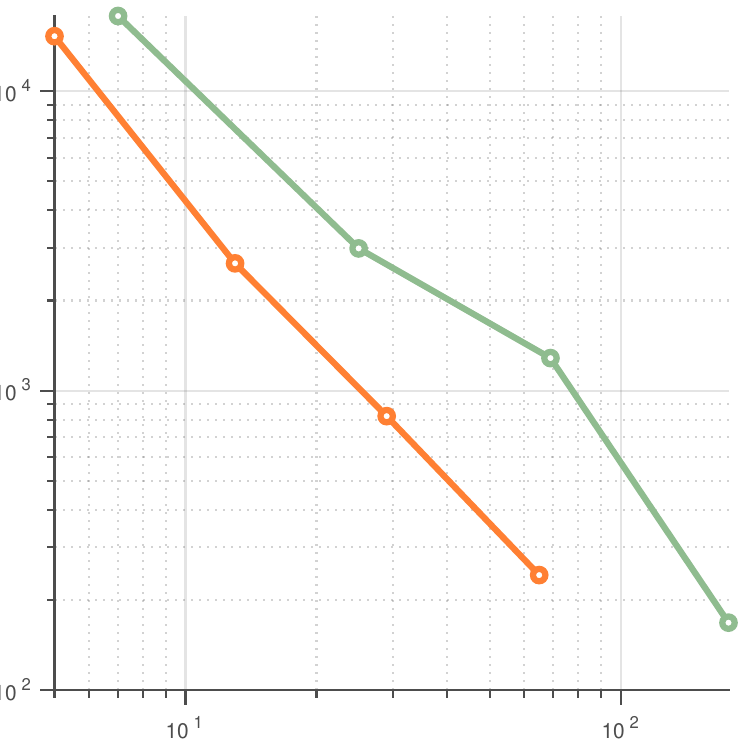}};
			%\node at (-2.5,-2.4) {(b)};
			%\node at (3,-2.4) {(c)};  
			\node [rotate = 90] at (-4.75,0.3) {\Large 
				$|\eset{Q(\hat u)} - \eset{\mcS^{m,g}_{w}[Q(\hat u)]}|$}; 
			\node at (1.25,-.4) {\Large N = 2};
			\node at (2,1) {\Large N = 3};
			\node at (0,-4) {\Large knots};
		\end{tikzpicture} \\
		(a) & (b)
	\end{tabular}
	\caption{Trypin protein NPBE test. (a) 
		Visual representation of the Trypsin molecule's secondary structure, rendered as a meshed surface that outlines the molecular boundary. This boundary is obtained by rolling the molecule with a solvent atom and is referred to as interface $I_1$. The potential field is computed with the
		APBS software for each knot in the sparse grid. (b) Convergence graphs of the error $|\eset{Q(\hat u)} - \eset{\mcS^{m,g}_{w}[Q(\hat u)]}|$ giving stochastic shifts of the molecular domain. We notice that the error decays
		algebraically.}
	\label{NumericalResults:Fig1}
\end{figure}

The protein will be now shifted using a stochastic model inside the domain $\mcD$. Since boundary
conditions are set to zero the solution will not be a simple translation, but will depend nonlinearly on the shift of the atom locations given by $\mcC$. For each shift, the domain of the Trypsin molecule is discretized and the potential field is solved using APBS. Let $\mcC(\omega)$ correspond to the set of atom locations shifted by the event outcome $\omega \in \Omega$, i.e.
\[
	\mcC(\omega) = \{\bx = x_0 + \sum_{k = 1}^{N} \alpha_{k} \bbe_{k}Y_{k}(\omega) \,|\, x_0 \in \mcC\},
\]
where
\[
	\bbe_1 = [1,0,0]^{T},  \bbe_2 = [0,1,0]^{T},  \bbe_3 = [0,0,1]^{T}, \alpha_1 = \alpha_2 = \alpha_3 = 10,
\]
$N \leq 3$, and $Y_1(\omega)$, $Y_2(\omega)$, $Y_3(\omega)$ are uniformly distributed  over the domain $[-\sqrt{3},\sqrt{3}]$  and independent of each other. Thus for each event
$\omega \in \Omega$ each element of $\mcC$ are shifted as follows:
\[
	\bx_k  \rightarrow \bx_k +  \sum_{k = 1}^{N} \alpha_{k} \bbe_{k}Y_{k}(\omega),
\]
for $k = 1,\dots, N$. Suppose $\hat u(x,Y_1,\dots,Y_N)$ is the solution of the APBS with 
respect to $Y_1,\dots,Y_N$ and let
\begin{equation}\label{QoI}
	Q(\hat u) := \int_{\mcD} \hat u(x,Y_1,\dots,Y_N)\,\mbox{d}x
\end{equation}
with respect to stochastic domain shifts $Y_1,\dots,Y_N$. Using the sparse grid the error $|\eset{Q(\hat u)} - \eset{\mcS^{m,g}_{w}[Q(\hat u)]}|$ is computed.

In Figure \ref{NumericalResults:Fig1} (b) the convergence graphs are plotted for $w = 2,3,4,5$ for $N = 2$. We assume that $w = 7$ is the actual value for $\eset{Q(\hat u)}$. Notice that the convergence rate decays algebraically. This is consistent with \cref{erroranalysis:theorem1}. We observe the similar algebraic decay for $N = 3$.

\section{Conclusion}

In this paper we show the existence of complex analytic extension of the solution of the NPBE with respect to the random perturbations of the domain, and its application to UQ is explored. This is a difficult problem due to the exponentially-growing nonlinear term and the discontinuity of the interface. The analyticity of solutions holds significant practical implications for efficiently computing quantities of interest. Any
bounded linear map $Q: \mcH(U) \rightarrow \bbR$
can be computed at an algebraic or sub-exponential rate, since they are analytic with respect to solutions \(\tilde{u} \in \mathcal{H}(U)\). Notably, the linearity of surface integrals allows for the direct application of these results in such cases.

We have also given estimates on the region of analyticity for the solution of the NPBE, which would allow for explicit calculations of the convergence rate of the sparse grid quadrature.  However the estimate of the region relies on constants, whose exact values are difficult to determine, a topic of interest that are left for future work.

More generally, the framework developed here is applicable to other problems in UQ. The strategy to applying these results to other problems is roughly as follows:
\begin{enumerate}[label = \arabic*.]
    \item Rewrite the problem as a functional equation, \(\mcF(\by, u) = 0\), where \(\by\) are the stochastic parameters in \(\bbR^N\) or \(\bbC^N\) and \(u\) represents the solution to the nonlinear PDE or equation in an appropriate Banach space. If the problem involves an interface, then a space similar to \(\mcH(U)\) may be useful.
    \item Applying the implicit function theorem (similar to how it was done for \cref{existence-of-region}) gives a region of analyticity for the solution \(\by \mapsto u(\by)\). It is important that the Banach space for \(u\) is appropriately chosen so that the Fr\'{e}chet derivative \(D_u \mcF\) is an isomorphism.
    \item \Cref{theorem-analytic-ift-region} can be applied to estimate the region of analyticity. This gives useful convergence results for numerical methods (e.g.\ sparse grid quadratures). To get explicit bounds, one needs to estimate the norm of the linear operator \([D_u F]^{-1}\), which for simple domains and operators can be computed.
\end{enumerate}

Thus the results in this paper can be more broadly implemented in other nonlinear UQ problems in order to demonstrate the analyticity of solutions and observables as well as determine estimates on the region of analyticity. 

\appendix

\section{Computations for estimates}\label{app:computations}

\begin{proof}[Proof of \cref{prop:computations}]
	The space \(\mcL\) is nice to work with since it is a Banach algebra, i.e., for \(A_1, A_2 \in \mcL\) we have \[\|A_1A_2\|_{\mcL} \leq \|A_1\|_\mcL \|A_2\|_\mcL.\] We will use this fact frequently in the proof.
	
	From \cref{holder}, we have that 
	\begin{equation*}
		\sigma_{\mathrm{max}} (\mcB \by (\br)) = |\mcB\by(\br)|_2 \leq \|\mcB \|_1 |\by|_\infty \quad \text{for any } \br\in U
	\end{equation*}
	where \(\sigma_{\mathrm{max}}(\cdot)\) denotes the largest singular value of a matrix. Thus (suppressing the \(\br\in U\) term)
	\begin{equation*}
		\begin{aligned}
			\sigma_{\mathrm{max}}(J(\by)) &= |J(\by)|_2 \leq 1 + \| \mcB\|_1 |\by|_\infty \\
			\sigma_{\mathrm{min}} (J(\by))  &\geq 1 - \sigma_{\mathrm{max}}(\mcB \by) \geq 1 - \|\mcB\|_1 |\by|_\infty,
		\end{aligned}
	\end{equation*}
	where \(\sigma_{\mathrm{min}}(\cdot)\) represents the smallest singular value of a matrix.
	
	We will need estimates on the norms \(J^{-1}(\by^0 + \by)\) and \(\det J(\by^0 + \by)\), and so the following identity will be useful:
	\begin{equation*}
		J(\by^0 + \by) = J(\by^0) [ I + J^{-1}(\by^0) \mcB\by].
	\end{equation*}
	By the assumption \cref{small-b}, we have that \(J(\by^0)\) is invertible, and the inverse is given by the Neumann series
	\begin{equation*}
		J^{-1}(\by^0) = I + \sum_{k=1}^\infty (-1)^k (\mcB \by^0)^k.
	\end{equation*}
	Thus
	\begin{align*}
		\| J^{-1}(\by^0) \|_{\mcL} &\leq 1 + \sum_{k=1}^\infty (\|\mcB \by^0\|_\mcL)^k \\
		&\leq 1 + \sum_{k=1}^\infty (4 \| \mcB \by^0 \|_{C^1})^k \\
		&= \frac{1}{1 - 4 \| \mcB\|_1 |\by^0|_\infty},
	\end{align*}
	which give \cref{Jinf-L}. \Cref{Jinf-L2} is shown similarly. Also,
	\begin{align*}
		\|(I + J(\by^0)\mcB\by)^{-1} - I \|_\mcL &\leq \sum_{k=1}^\infty \| J^{-1}(\by^0) \mcB\by\|_\mcL ^k \\
		&\leq \sum_{k=1}^\infty \left( \frac{1}{1 - 4 \| \mcB\|_1 |\by^0|_\infty} \right)^k (4 \|\mcB \|_1 |\by|_\infty)^k \\
		&= \frac{4 \| \mcB\|_1 |\by|_\infty}{1 - 4 \| \mcB \|_1 (|\by|_\infty + |\by^0|_\infty)},
	\end{align*}
	which gives \cref{I-J-B}.
	To estimate \(\|\det J(\by) \|_\mcL\), we will need to bound both \(|\det J(\by)|\) and \(|\frac{\partial}{\partial r_i} \det J (\by) |\) over \(U\). For a matrix \(A\) with \(\sigma_{\mathrm{max}}(A) < 1\), we have 
	\begin{equation*}
		\det(I+A) = 1 + \sum_{k=1}^\infty \frac 1 {k!} \left( - \sum_{j=1}^\infty \frac{(-1)^j}{j}\mathrm{tr}(A^j)\right)^k.
	\end{equation*}
	Therefore, 
	\begin{align*}
		|\det J(\by^0) | &= |\det (I + \mcB \by^0) | \\
		&\leq 1 + \sum_{k=1}^\infty \frac 1 {k!} \left(  \sum_{j=1}^\infty \frac{1}{j}\mathrm{tr}( (\mcB \by^0)^j)\right)^k \\
		&\leq 1 + \sum_{k=1}^\infty \frac 1 {k!} \left(  \sum_{j=1}^\infty \frac{3}{j}\sigma_{\mathrm{max}}( \mcB \by^0)^j\right)^k \\
		&\leq 1 + \sum_{k=1}^\infty \frac 1 {k!} \left(  \sum_{j=1}^\infty \frac{3}{j} \|\mcB\|_1^j |\by^0|_\infty^j\right)^k \\
		&= 1 + \sum_{k=1}^\infty \frac 1 {k!}(-3 \ln(1 - \|\mcB\|_1 |\by^0|_\infty))^k \\
		&= \frac{1}{(1- \|\mcB\|_1 |\by^0|)^3},
	\end{align*}
	which gives \cref{absdet}. The inequality in \cref{absdet2} is similarly shown.
	
	Using Jacobi's formula, we have that
	\begin{align*}
		\left|\frac{\partial}{\partial r_i} \det J(\by^0) \right| &\leq |\det J(\by^0)| \times \left|\mathrm{tr}(J^{-1}(\by^0)  \frac{\partial J}{\partial r_i}(\by^0))\right| \\
		&\leq \frac{1}{(1- \|\mcB\|_1 |\by^0|)^3} \times 3 \sigma_{\mathrm{min}}(J(\by^0))^{-1} \times \sigma_{\mathrm{max}} \left( \frac \partial {\partial r_i} \mcB\by^0\right) \\
		&\leq \frac{1}{(1- \|\mcB\|_1 |\by^0|)^3} \times \frac{3}{1 - \|\mcB\|_1 |\by^0|} \times \|\mcB\|_1 |\by^0|_\infty \\
		&= \frac{3 \|\mcB\|_1 |\by^0|_\infty}{(1- \|\mcB\|_1 |\by^0|)^4}
	\end{align*}
	Therefore,
	\begin{align*}
		\|\det J(\by^0) \|_\mcL &\leq 4 \| \det J(\by^0) \|_{C^1} \\
		&\leq 4 \max\left\{ \frac{1}{(1-\|\mcB\|_1 |\by^0|_\infty)^3}, \frac{3 \|\mcB\|_1 |\by^0|_\infty}{(1- \|\mcB\|_1 |\by^0|)^4}\right\} \\
		&= \frac{4}{(1-\|\mcB\|_1 |\by^0|_\infty)^3}
	\end{align*}
	where the last line follows from our assumptions \cref{small-b} and \(\by^0 \in \Gamma = [-1,1]^N\). This gives \cref{det-L}, and \cref{det-L2} is proved similarly.
	
	Lastly, we want to get the estimate \cref{det-I-J-B}. We have that
	\begin{align*}
		|\det(I + J^{-1}(\by^0)\mcB \by) -1 | &\leq \sum_{k=1}^\infty \frac 1 {k!} \left( \frac 1 j |\mathrm{tr}\left((J^{-1}(\by^0)\mcB \by\right)^j)|\right)^k \\
		&\leq \sum_{k=1}^\infty \frac 1 {k!} \left( \sum_{j=1}^\infty \frac 3 j \left( \frac{\|\mcB\|_1 |\by|_\infty}{1 - \|\mcB\|_1 |\by^0|_\infty} \right)^j \right)^k \\
		&= \left(\frac{1 - \|\mcB\|_1 |\by^0|}{1 - \|\mcB\|_1(|\by^0|_\infty + |\by|_\infty)}\right)^3 - 1.
	\end{align*}
	Using Jacobi's formula, we have that the corresponding derivative can be bounded as follows:
	\begin{align*}
		&\left|\frac \partial {\partial r_i} \det (I + J^{-1}(\by^0)\mcB\by) \right| \\
		&\quad \leq |\det(I + J^{-1}(\by^0)\mcB \by)| \times \left| \mathrm{tr} \left(  (I + J^{-1}(\by^0)\mcB\by)^{-1} \left(\frac{\partial J^{-1}(\by^0)}{\partial r_i} \mcB \by + J^{-1}(\by^0) \frac{\partial \mcB\by}{\partial r_i}\right) \right) \right|.
	\end{align*}
	Note that
	\begin{align*}
		\sigma_{\mathrm{min}}(I + J^{-1}(\by^0) \mcB\by) &= \sigma_{\mathrm{min}}(J^{-1}(\by^0) (I + \mcB\by^0 + \mcB\by)) \\
		&\geq \sigma_{\mathrm{min}}(J^{-1}(\by^0)) \sigma_{\mathrm{min}}(I + \mcB\by^0 + \mcB\by) \\
		&\geq \frac{1 - \|\mcB\|_1(|\by^0| + |\by|_\infty)}{1 + \|\mcB\|_1 |\by^0|_\infty}.
	\end{align*}
	We also have that 
	\begin{equation*}
		\frac{\partial J^{-1}(\by^0)}{\partial r_i} = - J^{-1}(\by^0) \frac{\partial \mcB\by^0}{\partial r_i}J^{-1}(\by^0) 
	\end{equation*}
	so that 
	\begin{equation*}
		\sigma_{\mathrm{max}} \left( \frac{\partial J^{-1}(\by^0)}{\partial r_i} \right) \leq \frac{\|\mcB\|_1 |\by^0|_\infty}{(1-\|\mcB\|_1 |\by^0|_\infty)^2}.
	\end{equation*}
	Therefore,
	\begin{align*}
		&\left|\frac{\partial}{\partial r_i} \det (I + J^{-1}(\by^0) \mcB\by)\right| \\
		&\quad\leq \left[ \left(\frac{1 - \|\mcB\|_1 |\by^0|}{1 - \|\mcB\|_1(|\by^0|_\infty + |\by|_\infty)}\right)^3 - 1 \right] \times 3 \times \frac{1 + \|\mcB\|_1 |\by^0|_\infty}{1 - \|\mcB\|_1(|\by^0| + |\by|_\infty)} \\
		&\qquad\times \left[\frac{\|\mcB\|_1 |\by^0|_\infty}{(1-\|\mcB\|_1 |\by^0|_\infty)^2} \|\mcB\|_1 |\by^0|_\infty + \frac{\|\mcB\|_1|\by^0|_\infty}{1 - \|\mcB\|_1 |\by^0|_\infty}\right]
	\end{align*}
	It is possible to show that 
	\begin{equation*}
		\left|\frac{\partial}{\partial r_i} \det (I + J^{-1}(\by^0) \mcB\by)\right| \leq  \left(\frac{1 - \|\mcB\|_1 |\by^0|}{1 - \|\mcB\|_1(|\by^0|_\infty + |\by|_\infty)}\right)^3 - 1 
	\end{equation*}
	using the assumptions of the norms of \(\by\), \(\by^0\), and \(\mcB\). Thus
	\begin{align*}
		\|\det (I + J^{-1}(\by^0) \mcB\by) \|_\mcL &\leq \|\det (I + J^{-1}(\by^0) \mcB\by)\|_{C^1} \\
		&\leq 4 \left[ \left(\frac{1 - \|\mcB\|_1 |\by^0|}{1 - \|\mcB\|_1(|\by^0|_\infty + |\by|_\infty)}\right)^3 - 1  \right].
	\end{align*}
\end{proof}

\bibliographystyle{plain}
\bibliography{ref_nPBE,citations,nPBEapplicationsALL}

\begin{thebibliography}{10}

\bibitem{adams2003sobolev}
Robert~A Adams and John~JF Fournier.
\newblock {\em Sobolev spaces}.
\newblock Elsevier, 2003.

\bibitem{Allen2004}
M.~Allen.
\newblock Introduction to molecular dynamics simulation, lecture notes.
\newblock In N~Attig, K~Binder, H~Grubmuller, and K~Kremer, editors, {\em
  Computational Soft Matter: From Synthetic Polymers to Proteins, NIC Series},
  volume~23, pages 1--28. John von Neumann Institute for Computing, Julich,
  2004.

\bibitem{Ashcroft1976}
N~W Ashcroft and N~D Mermin.
\newblock {\em Solid State Physics}.
\newblock Holt Rinehart and Winston, 1976.

\bibitem{Atkins1986}
P~W Atkins.
\newblock {\em Physical Chemistry}.
\newblock Freeman, 1986.

\bibitem{Back2011}
J.~B\"{a}ck, F.~Nobile, L.~Tamellini, and R.~Tempone.
\newblock Stochastic spectral {G}alerkin and collocation methods for {PDE}s
  with random coefficients: A numerical comparison.
\newblock In Jan~S. Hesthaven and Einar~M. Rønquist, editors, {\em Spectral
  and High Order Methods for Partial Differential Equations}, volume~76 of {\em
  Lecture Notes in Computational Science and Engineering}, pages 43--62.
  Springer Berlin Heidelberg, 2011.

\bibitem{Baker2001}
NA~Baker, D~Sept, S~Joseph, MJ~Holst, and JA~McCammon.
\newblock Electrostatics of nanosystems: application to microtubules and the
  ribosome.
\newblock {\em Proceedings of the National Academy of Sciences of the United
  States of America}, 98(18):10037—10041, August 2001.

\bibitem{Barthel1998}
J~M~G Barthel, H~Krienke, and W~Kunz.
\newblock {\em Physical Chemistry of Electrolyte Solutions}.
\newblock Springer, 1998.

\bibitem{Novak_Ritter_00}
V.~Barthelmann, E.~Novak, and K.~Ritter.
\newblock High dimensional polynomial interpolation on sparse grids.
\newblock {\em Advances in Computational Mathematics}, 12:273--288, 2000.

\bibitem{Bayliss1959}
L~E Bayliss.
\newblock {\em Principles of General Physiology}, volume vol 1.
\newblock Longmans, 1959.

\bibitem{Bergethon1998}
P~R Bergethon.
\newblock {\em The Physical Basis of Biochemistry}.
\newblock Springer, 1998.
\newblock DH for electrolyte solutions and p 419 PB for double layer.

\bibitem{Berman2000}
H.~M. Berman, J.~Westbrook, Z.~Feng, G.~Gilliland, T.N. Bhat, H.~Weissig, I.N.
  Shindyalov, and P.E. Bourne.
\newblock The protein data bank (www.pdb.org).
\newblock {\em Nucleic Acids Res.}, 28:235--242, 2000.

\bibitem{Berry2000}
R~S Berry, S~A Rice, and J~Ross.
\newblock {\em Physical Chemistry}.
\newblock Oxford University Press, 2000.
\newblock PB and DH for electrolyte solutions and p 790 PB for double layer.

\bibitem{Blum1992}
L~Blum and D~Henderson.
\newblock {\em Fundamentals of Inhomogeneous Fluids}.
\newblock Marcel Dekker, 1992.

\bibitem{Bockris1970}
J~O’~M Bockris and A~K~N Reddy.
\newblock {\em Modern Electrochemistry}, volume vol 1 and 2.
\newblock Plenum, 1970.
\newblock PB and DH for electrolyte solutions and p 722 PB for double layer.

\bibitem{Brout1963}
R~Brout and P~Carruthers.
\newblock {\em Lectures on the Many-Electron Problem}.
\newblock Interscience, 1963.
\newblock Classical PB and DH for electron gas and p 99 quantum TF.

\bibitem{Butt2013}
H-J Butt, K~Graf, and M~Kappl.
\newblock {\em Physics and Chemistry of Interfaces}.
\newblock Wiley-VCH, 2013.

\bibitem{Cai2013}
W~Cai.
\newblock {\em Computational Methods for Electromagnetic Phenomena}.
\newblock Cambridge University Press, 2013.
\newblock ch 2 PB and DH theory and ch 3 numerical methods for solving the PB
  equation.

\bibitem{castrillon2016}
J.~E. Castrill\'{o}n-Cand\'{a}s, F.~Nobile, and R.~Tempone.
\newblock Analytic regularity and collocation approximation for {PDEs} with
  random domain deformations.
\newblock {\em Computers and Mathematics with applications}, 71(6):1173--1197,
  2016.

\bibitem{castrillon2016analytic}
Julio~E Castrillon-Candas, Fabio Nobile, and Raul~F Tempone.
\newblock Analytic regularity and collocation approximation for elliptic {PDE}s
  with random domain deformations.
\newblock {\em Computers \& Mathematics with Applications}, 71(6):1173--1197,
  2016.

\bibitem{castrillon2021hybrid}
Julio~E Castrill{\'o}n-Cand{\'a}s, Fabio Nobile, and Ra{\'u}l~F Tempone.
\newblock A hybrid collocation-perturbation approach for {PDE}s with random
  domains.
\newblock {\em Advances in Computational Mathematics}, 47(3):1--35, 2021.

\bibitem{castrillon2021stochastic}
Julio~E Castrill{\'o}n-Cand{\'a}s and Jie Xu.
\newblock A stochastic collocation approach for parabolic {PDE}s with random
  domain deformations.
\newblock {\em Computers \& Mathematics with Applications}, 93:32--49, 2021.

\bibitem{Ceriotti2009}
M.~Ceriotti, G.~Bussi, and M.~Parrinello.
\newblock Langevin equation with colored noise for constant-temperature
  molecular dynamics simulations.
\newblock {\em Physical Review Letters}, 102(2):020601, 2009.

\bibitem{Chaikin1995}
P~M Chaikin and T~C Lubensky.
\newblock {\em Principles of Condensed Matter Physics}.
\newblock Cambridge University Press, 1995.

\bibitem{chang2003analytic}
Hung-Chieh Chang, Wei He, and Nagabhushana Prabhu.
\newblock The analytic domain in the implicit function theorem.
\newblock {\em JIPAM. J. Inequal. Pure Appl. Math}, 4(1), 2003.

\bibitem{Chavazaviel1999}
J-N Chavazaviel.
\newblock {\em Coulomb Screening by Mobile Charges. Applications to Materials
  Science, Chemistry and Biology}.
\newblock Birkhaűser, 1999.
\newblock Classical PB and p 25 quantum TF for electrons in metal.

\bibitem{choi2021existence}
Brian Choi, Jie Xu, Trevor Norton, Mark Kon, and Julio~E. Castrillon-Candas.
\newblock Existence of strong solution for the complexified non-linear poisson
  boltzmann equation, 2021.

\bibitem{Dean2014}
D~Dean, J~Dobnikar, A~Naji, and R~Podgornik, editors.
\newblock {\em Electrostatics of Soft and Disordered Matter}.
\newblock Pan Stanford Publishing, 2014.

\bibitem{Doi2016}
M~Doi.
\newblock {\em Soft Matter Physics}.
\newblock Oxford University Press, 2016.

\bibitem{Edsall1958}
J~T Edsall and J~Wyman.
\newblock {\em Biophysical Chemistry}.
\newblock Academic, 1958.

\bibitem{Evans1999}
D~F Evans and H~Wennerstrom.
\newblock {\em The Colloidal Domain, where Physics, Chemistry and Biology
  Meet}.
\newblock Wiley-VCH, 1999.

\bibitem{Fawcett2004}
W~R Fawcett.
\newblock {\em Liquids, Solutions and Interfaces}.
\newblock Oxford University Press, 2004.
\newblock DH for electrolytes, p 542 PB for planar double layer.

\bibitem{Frenkel1946}
J~Frenkel.
\newblock {\em Kinetic Theory of Liquids}.
\newblock Oxford University Press, 1946.
\newblock The counter-ions only case was earlier discussed by.

\bibitem{Friedman1962}
H~C Friedman.
\newblock {\em Ionic Solution Theory}.
\newblock Interscience, 1962.

\bibitem{Giuliani2003}
G~F Giuliani and G~Vignale.
\newblock {\em Quantum Theory of the Electron Liquid}.
\newblock Cambridge University Press, 2003.

\bibitem{Glasstone1947}
S~Glasstone.
\newblock {\em Thermodynamics for Chemists}.
\newblock Van Nostrand, 1947.
\newblock (reprinted by Krieger 1972).

\bibitem{Gray2018}
C~G Gray and P~J Stiles.
\newblock Nonlinear electrostatics: the poisson{\textendash}boltzmann equation.
\newblock {\em European Journal of Physics}, 39(5):053002, jul 2018.

\bibitem{Hanggi1995}
P.~H\"anggi and P.~Jung.
\newblock {\em Colored Noise in Dynamical Systems}, pages 239--326.
\newblock John Wiley \& Sons, Inc., 2007.

\bibitem{Hansen2013}
J-P Hansen and I~R McDonald.
\newblock {\em Theory of Simple Liquids, with Applications to Soft Matter}.
\newblock Elsevier, 2013.
\newblock DH for electrolyte solutions, p 413 DH for plasmas and other
  classical ionic fluids, p 435 PB for double layer, p 448 TF for electron
  fluids.

\bibitem{Heitzinger2018}
Clemens Heitzinger, Michael Leumüller, Gudmund Pammer, and Stefan Rigger.
\newblock Existence, uniqueness, and a comparison of nonintrusive methods for
  the stochastic nonlinear poisson--boltzmann equation.
\newblock {\em SIAM/ASA Journal on Uncertainty Quantification},
  6(3):1019--1042, 2018.

\bibitem{Hobbie1988}
R~K Hobbie.
\newblock {\em Intermediate Physics for Medicine and Biology}.
\newblock Wiley, 1988.

\bibitem{Hober1947}
R~Hober.
\newblock {\em Physical Chemistry of Cells and Tissues}.
\newblock J and A Churchill, 1947.

\bibitem{Holm2001}
C~Holm, P~Kekcheff, and R~Podgornik, editors.
\newblock {\em Electrostatic Effects in Soft Matter and Biophysics}.
\newblock Kluwer, 2001.

\bibitem{holst1994poisson}
Michael~J Holst.
\newblock The {P}oisson-{B}oltzmann equation: Analysis and multilevel numerical
  solution.
\newblock {\em Applied Mathematics and CRPC, California Institute of
  Technology}, 1994.

\bibitem{Hunter2001}
R~J Hunter.
\newblock {\em Foundations of Colloid Science}.
\newblock Oxford University Press, 2001.

\bibitem{Ichimaru1984}
S~Ichimaru.
\newblock {\em Statistical Plasma Physics, Vol. II, Condensed Plasmas}.
\newblock Addison-Wesley, 1984.
\newblock Quantum plasmas, p 18 classical plasmas, p 116 electron liquids.

\bibitem{Israelachvili1991}
J~Israelachvili.
\newblock {\em Intermolecular and Surface Forces, with Applications to Colloid
  and Biological Systems}.
\newblock Academic, 1991.

\bibitem{Jung1987}
P.~Jung and P.~H\"anggi.
\newblock Dynamical systems: A unified colored-noise approximation.
\newblock {\em Phys. Rev. A}, 35:4464--4466, May 1987.

\bibitem{jurrus2018improvements}
Elizabeth Jurrus, Dave Engel, Keith Star, Kyle Monson, Juan Brandi, Lisa~E
  Felberg, David~H Brookes, Leighton Wilson, Jiahui Chen, Karina Liles, et~al.
\newblock Improvements to the apbs biomolecular solvation software suite.
\newblock {\em Protein Science}, 27(1):112--128, 2018.

\bibitem{Kirkwood1961}
J~G Kirkwood and I~Oppenheim.
\newblock {\em Chemical Thermodynamics}.
\newblock McGraw-Hill, 1961.

\bibitem{Kittel1996}
C~Kittel.
\newblock {\em Introduction to Solid State Physics}.
\newblock Wiley, 1996.
\newblock TF screening of a positive charge placed in a sea of electrons.

\bibitem{Landau1958}
L~D Landau and E~M Lifshitz.
\newblock {\em Statistical Physics}.
\newblock Addison-Wesley, 1958.
\newblock Classical PB and DH and p 232 quantum TF.

\bibitem{Lewis1961}
G~N Lewis and M~Randall.
\newblock {\em Thermodynamics}.
\newblock McGraw-Hill, 1961.

\bibitem{Li2009}
B~Li.
\newblock Minimization of the electrostatic free energy and the
  poisson–boltzmann equation for molecular solvation with implicit solvent.
\newblock {\em SIAM J. Math. Anal.}, 40:2536--2566, 2009.

\bibitem{Lyklema1991}
J~Lyklema.
\newblock {\em Fundamentals of Interface and Colloid Science}, volume vol 1 and
  2.
\newblock Academic, 1991.
\newblock 1995.

\bibitem{March1984}
N~H March and M~P Tosi.
\newblock {\em Coulomb Liquids}.
\newblock Academic, 1984.
\newblock Classical DH for bulk ionic fluids, p 42 quantum TF for electron
  fluids, p 262 classical PB for double layer.

\bibitem{McKelvey1966}
J~P McKelvey.
\newblock {\em Solid-State and Semiconductor Physics}.
\newblock Harper and Row, 1966.
\newblock For applications of double layers to (a) semiconductors, see
  Chavazaviel [67].

\bibitem{mclean2000strongly}
William McLean and William Charles~Hector McLean.
\newblock {\em Strongly elliptic systems and boundary integral equations}.
\newblock Cambridge university press, 2000.

\bibitem{McQuarrie1976}
D~A McQuarrie.
\newblock {\em Statistical Mechanics}.
\newblock Harper and Row, 1976.

\bibitem{Muthukumar2011}
M~Muthukumar.
\newblock {\em Polymer Translocation}.
\newblock CRC Press, 2011.

\bibitem{Nattino2019}
F.~Nattino, M.~Truscott, N.~Marzari, and O.~Andreussi.
\newblock Continuum models of the electrochemical diffuse layer in
  electronic-structure calculations.
\newblock {\em J. Chem. Phys.}, 150:041722, 2019.

\bibitem{Neumaier1997}
A.~Neumaier.
\newblock Molecular modeling of proteins and mathematical prediction of protein
  structure.
\newblock {\em SIAM Rev.}, 39(3):407–460, September 1997.

\bibitem{nobile2008b}
F.~Nobile, R.~Tempone, and C.~Webster.
\newblock An anisotropic sparse grid stochastic collocation method for partial
  differential equations with random input data.
\newblock {\em SIAM Journal on Numerical Analysis}, 46(5):2411--2442, 2008.

\bibitem{nobile2008a}
F.~Nobile, R.~Tempone, and C.~Webster.
\newblock A sparse grid stochastic collocation method for partial differential
  equations with random input data.
\newblock {\em SIAM Journal on Numerical Analysis}, 46(5):2309--2345, 2008.

\bibitem{Ohshima2010}
H~Ohshima.
\newblock {\em Biophysical Chemistry of Biointerfaces}.
\newblock Wiley, 2010.
\newblock PB and DH for planar, spherical and cylindrical double layers.

\bibitem{Opschoor2021}
J.~A.~A. Opschoor, Ch. Schwab, and J.~Zech.
\newblock Exponential relu dnn expression of holomorphic maps in high
  dimension.
\newblock {\em Constructive Approximation}, Apr 2021.

\bibitem{Poon2006}
W~K~C Poon and D~Andelman, editors.
\newblock {\em Soft Condensed Matter Physics in Molecular and Cell Biology}.
\newblock Taylor and Francis, 2006.

\bibitem{Rubinstein1990}
I~Rubinstein.
\newblock {\em Electro-Diffusion of Ions}.
\newblock Society for Industrial and Applied Mathematics, 1990.

\bibitem{scarabosio2022deep}
Laura Scarabosio.
\newblock Deep neural network surrogates for nonsmooth quantities of interest
  in shape uncertainty quantification.
\newblock {\em SIAM/ASA Journal on Uncertainty Quantification},
  10(3):975--1011, 2022.

\bibitem{Schillings2013}
Claudia Schillings and Christoph Schwab.
\newblock {Sparse, adaptive Smolyak quadratures for Bayesian inverse problems}.
\newblock {\em Inverse Problems}, 29(6):065011, 2013.

\bibitem{Schmickler2010}
W~Schmickler.
\newblock {\em Interfacial Electrochemistry}.
\newblock Oxford University Press, 2010.

\bibitem{Smolyak63}
S.~Smolyak.
\newblock {Quadrature and interpolation formulas for tensor products of certain
  classes of functions}.
\newblock {\em Soviet Mathematics, Doklady}, 4:240--243, 1963.

\bibitem{Sneppen2005}
K~Sneppen and G~Zocchi.
\newblock {\em Physics in Molecular Biology}.
\newblock Cambridge University Press, 2005.

\bibitem{Sparnaay1972}
M~J Sparnaay.
\newblock {\em The Double Layer}.
\newblock Pergamon, 1972.

\bibitem{Stein2019}
Christopher~J. Stein, John~M. Herbert, and Martin Head-Gordon.
\newblock The poisson--boltzmann model for implicit solvation of electrolyte
  solutions: Quantum chemical implementation and assessment via sechenov
  coefficients.
\newblock {\em The Journal of Chemical Physics}, 151(22):224111, Dec 2019.

\bibitem{Sundararaman2018}
R.~Sundararaman, K.~Letchworth-Weaver, and K.~A. Schwarz.
\newblock Improving accuracy of electrochemical capacitance and solvation
  energetics in first-principles calculations.
\newblock {\em J. Chem. Phys.}, 148:144105, 2018.

\bibitem{Sundararaman2017}
R.~Sundararaman and K.~Schwarz.
\newblock Evaluating continuum solvation models for the electrode-electrolyte
  interface: Challenges and strategies for improvement.
\newblock {\em J. Chem. Phys.}, 146:084111, 2017.

\bibitem{Morrison1967}
N~G van Kampen and B~U Felderhof.
\newblock {\em Theoretical Methods in Plasma Physics}.
\newblock North-Holland, 1967.

\bibitem{Verwey1948}
E~J~W Verwey and J~T~G Overbeek.
\newblock {\em Theory of the Stability of Lyophobic Colloids}.
\newblock Elsevier, 1948.
\newblock (reprinted by Dover 1999) This is the classic monograph on double
  layers in electrolyte solutions. PB theory starts on p 22.

\bibitem{whittlesey1965analytic}
Emmet~F Whittlesey.
\newblock Analytic functions in banach spaces.
\newblock {\em Proceedings of the American Mathematical Society},
  16(5):1077--1083, 1965.

\bibitem{Xia2013}
K.~Xia and G.~Wei.
\newblock Stochastic model for protein flexibility analysis.
\newblock {\em Phys. Rev. E}, 88:062709, Dec 2013.

\end{thebibliography}
		
\end{document}